\documentclass[11pt,a4paper]{amsart}

\usepackage{amssymb, amsthm, amsmath, amsfonts, mathrsfs, dsfont, stmaryrd}
\usepackage{mathtools}

\usepackage{tikz-cd}
\usetikzlibrary{
    calc, 3d, patterns, arrows.meta, decorations.pathmorphing,
    decorations.pathreplacing, through, shapes, matrix
}

\tikzset{
    labl/.style={anchor=south, rotate=-90, inner sep=.6mm},
    symbol/.style={draw=none, every to/.append style={edge node={node [sloped, allow upside down, auto=false]{$#1$}}}},
    ext/.style={circle, draw, inner sep=1pt},
    int/.style={circle, draw, fill, inner sep=1pt},
    exte/.style={circle, draw, inner sep=3pt},
    inte/.style={circle, draw, fill, inner sep=3pt},
    diagram/.style={matrix of math nodes, row sep=3em, column sep=2.5em, text height=1.5ex, text depth=0.25ex},
    diagram2/.style={matrix of math nodes, row sep=0.5em, column sep=0.5em, text height=1.5ex, text depth=0.25ex},
    every picture/.style={baseline=-.65ex},
    every loop/.style={draw},
    rloop/.style={out=10, in=-10, loop, distance=3em},
    aloop/.style={out=100, in=80, loop, distance=3em}
}
\newcommand{\circled}[1]{%
	\tikz[baseline=(char.base)]\node[draw,circle,inner sep=0.6pt,minimum size=1.3em] (char){#1};}

\usepackage[left=1.5cm, right=1.5cm, top=3cm, bottom=3cm]{geometry}
\usepackage{indentfirst}
\usepackage{graphicx}
\usepackage{wrapfig, scalerel}
\usepackage{enumitem}

\usepackage{color}
\definecolor{darkblue}{rgb}{0,0.1,.5}
\usepackage[colorlinks=true, linkcolor=darkblue, urlcolor=darkblue, citecolor=darkblue]{hyperref}

\usepackage[labelformat=empty]{caption}

\makeatletter
\newcommand*{\RN}[1]{\expandafter\@slowromancap\romannumeral #1@}
\makeatother

\newcommand{\sbullet}[1][.5]{\mathbin{\vcenter{\hbox{\scalebox{#1}{$\bullet$}}}}}

\newcommand*\rcolon{%
    \nobreak
    \mskip6mu plus1mu
    \mathpunct{}%
    \nonscript
    \mkern-\thinmuskip
    {:}%
    \mskip2mu
    \relax
}

\newcommand{\hr}[2][]{\hyperref[#2]{#1~\ref{#2}}}

\def\R{\mathbb{R}}

\def\Z{\mathbb{Z}}
\def\Q{\mathbb{Q}}

\def\Der{\mathop{\mathrm{Der}}\nolimits}
\def\mor{\mathop{\mathrm{Mor}}\nolimits}

\def\emb{\mathop{\mathrm{Emb}}\nolimits}

\def\imm{\mathop{\mathrm{Imm}}\nolimits}

\def\path{\mathop{\mathrm{Path}}\nolimits}
\def\map{\mathop{\mathrm{Map}}\nolimits}
\def\mor{\mathop{\mathrm{Mor}}\nolimits}

\def\hofib{\mathop{\mathrm{hofib}}\nolimits}
\def\res{\mathop{\mathrm{Res}}\nolimits}
\def\cores{\mathop{\mathrm{coRes}}\nolimits}

\def\coind{\mathop{\mathrm{coInd}}\nolimits}
\def\mc{\mathop{\mathrm{MC}}_{\sbullet}\nolimits}
\def\hgc{\mathop{\mathrm{HGC}}\nolimits}
\def\Com{\mathsf{Com}}

\newcommand{\Graphs}{\mathsf{Graphs}}

\newcommand{\Gcat}{\mathsf{Graph}\mathcal Cat}

\def\MC{\mathop{\mathsf{MC}}\nolimits}

\def\id{\mathds{1}}

\def\sgn{\mathop{\mathrm{sgn}}}

\def\geq{\geqslant}

\newtheorem{theorem}{Theorem}[section]
\newtheorem*{theorem*}{Theorem}
\newtheorem{lemma}[theorem]{Lemma}

\newtheorem{proposition}[theorem]{Proposition}
\newtheorem{corollary}[theorem]{Corollary}
\theoremstyle{definition}
\newtheorem{construction}[theorem]{Construction}

\newtheorem{remark}[theorem]{Remark}
\numberwithin{equation}{section}

\newtheorem*{futuretheoreminner}{Theorem~\thefuturetheoreminner}
\newcommand{\thefuturetheoreminner}{} 

\ExplSyntaxOn

\prop_new:N \g_alevel_future_prop

\NewDocumentEnvironment{futuretheorem}{ m o +b }
{
    \renewcommand{\thefuturetheoreminner}{\ref{#1}}
    \IfNoValueTF{#2}
    {\futuretheoreminner}
    {\futuretheoreminner[#2]}
    {\emph{#3}}
}
{
    \endfuturetheoreminner
    \prop_gput:Nnn \g_alevel_future_prop { #1 } { #3 }
    \IfValueT{#2}{ \prop_gput:Nnn \g_alevel_future_prop { #1-attr } { #2 } }
}

\NewDocumentCommand{\pasttheorem}{m}
{
    \prop_if_in:NnTF \g_alevel_future_prop { #1-attr }
    {
        \begin{theorem}[\prop_item:Nn \g_alevel_future_prop { #1-attr }]
    }
    {
        \begin{theorem}
    }
    \label{#1}
    \prop_item:Nn \g_alevel_future_prop { #1 }
    \end{theorem}
}

\ExplSyntaxOff

\renewenvironment{proof}[1][{\itshape Proof}]{{\itshape #1. }}{\qed}

\title{Graphical composition of mapping spaces between modules of configuration-space-type}

\author{Semyon Abramyan}
\email{semyon.abramyan@gmail.com}

\begin{document}

\begin{abstract}
	In embedding calculus, spaces of embeddings are identified with derived mapping spaces between framed Fulton-MacPherson–type modules (framed configuration spaces). Unfortunately, there are no sufficiently good algebraic models for framed Fulton-MacPherson modules that would allow us to explicitly describe the rational homotopy type of the embedding space. Recently there were several attempts to avoid dealing with the framed versions of Fulton-MacPherson modules by considering framed manifolds, e.g. embeddings modulo immersions $\overline{\emb}$ in \cite{f-t-w20}, or embeddings with a deformation of the framing $\widetilde{\emb}$ in \cite{abra25}. In both cases, the rational homotopy type of the corresponding embedding space has an explicit description in terms of graphs (hairy graph complexes).
	
	We construct a combinatorial graphical composition for the composition of embedding spaces $\widetilde{\emb}$. As a step on our way, we describe the action of the coinduction functor on configuration-space-type $e_n^c$-comodules.
\end{abstract}
	
\maketitle
\tableofcontents

\section{Introduction}
In recent years a deep connection between the manifold calculus of functors introduced by Weiss~\cite{weis99} and the operad of little discs $L\mathbb D_m$ was established. Namely, convergence results of Goodwillie-Weiss~\cite{go-we99} (see also~\cite[Theorem~2.1]{turc13}) imply that for a pair of $m$-framed manifolds $M^m$ and $N^n$ with $\dim N - \dim M \geqslant 3$ there are weak equivalences
\[
\emb(M, N) \xrightarrow{\sim}
T_\infty\emb(M, N) \xrightarrow{\sim} \map^h_{\mathrm{mod}\text-\mathcal F_m}(\mathcal F_M, \mathcal F^{m\text{-fr}}_N),
\]
where $\mathcal F_m$, $\mathcal F_M$ and $\mathcal F^{m\text{-fr}}_N$ refer to the Fulton-MacPherson operad of $\R^m$, the Axelrod-Singer-Fulton-MacPherson completion of the configuration space of points on $M$ and its $m$-framed version on~$N$, respectively.
A different (though very similar) incarnation of the embedding functor was studied in \cite{ar-tu14, f-t-w20}. Namely, it was shown that for the embeddings modulo immersions functor
\[
\overline\emb(M, \R^n) \coloneqq \hofib \big(\emb(M, \R^n) \to \imm(M, \R^n)\big)
\]
there are weak equivalences
\[
\overline\emb(M, \R^n) \xrightarrow{\sim}
T_\infty\overline\emb(M, \R^n) \xrightarrow{\sim}
\map^h_{\mathrm{mod}\text-\mathcal F_m}(\mathcal F_M, \mathcal F_n).
\]
The rational homotopy type of the latter space can be described purely combinatorially.

Finally, in the recent preprint~\cite{abra25}, we introduced another embedding functor~$\widetilde{\emb}(M, N)$ for a pair of parallelized manifolds. The space $\widetilde{\emb}(M, N)$ consists of pairs $(f, h)$ of an embedding $f\colon M \to N$ and a deformation~$h$ of the frames from $df$ to the given frame on $N$. Convergence results similar to the above have been proved. Namely, there are weak equivalences
\[
	\widetilde{\emb}(M, N)
	\xrightarrow{\sim}
	T_\infty\widetilde{\emb}(M,N)
	\xrightarrow{\sim}
	\map^h_{\mathrm{mod}-\mathcal F_m}(\mathcal F_M, \mathcal F_N).
\]
Moreover, the following weak equivalence has been proved
\[
	\map^h_{\mathrm{mod}-\mathcal F_m} (\mathcal F_M, \mathcal F_N^\Q)
	\simeq
	\MC_{\sbullet}(\hgc_{A_M, H^{\sbullet}(N), n}^Z),
\]
where $\hgc$ is the hairy graph complex, $A_M$ is a model for $M$, and $n$ is the dimension of $N$ (see \hr[Remark]{remark:hgc for algebras} and \eqref{eq:hgc for algebras}).

This leads us to the following question: how to model "graphically" (combinatorially) the morphism of mapping spaces induced by composition of embeddings
\begin{equation}\label{eq:emb-tilda composition}
	\widetilde{\emb}(M,N) \times \widetilde{\emb}(N, K) \to \widetilde{\emb}(M,K) ?
\end{equation}

In order to answer the question above, in \hr[Section]{sec: graphical category} we introduce a graphical category $\Gcat_n$, whose objects are pairs $(U, Z_u)$ of a graded vector space $U$ and a Maurer-Cartan element $Z_u \in \hgc_{U, n}$, and the mapping space is given by
\[
\map_{\sbullet}\Big((U, Z_u), (V, Z_v)\Big) \coloneqq \MC_{\sbullet}(\hgc_{U,V,n}^{Z_u, Z_v}).
\]
The proof of the fact that there is a natural well-defined composition is the content of \hr[Section]{sec:graphical composition}.

The following is the main theorem of the paper, that, in particular, gives a combinatorial description of the composition above.

\begin{futuretheorem}{main}\label{thm:main-composition}
	There exists a simplicial functor
	\[
		\mathcal F\colon \Gcat_n \to \mathrm{dg\Lambda\!HopfMod}^{c\, op}_{e_n^c},
	\]
	such that
	\[
	\mathcal F\colon (U, Z_u) \mapsto \Graphs_{U,n}^{Z_u}
	\]
	and the composite 
	\begin{equation*}
	\map_{\Gcat_n}\big((U,Z_u), (V, Z_v)\big) \to
	\map_n(\Graphs_{V,n}^{Z_v}, \Graphs_{U,n}^{Z_u})
	\to
	\map^h_n(\Graphs_{V,n}^{Z_v}, \Graphs_{U,n}^{Z_u})
	\end{equation*}
	is a weak equivalence of simplicial sets.
\end{futuretheorem}

By \cite{will23}, the essential image of $\mathcal F$ can be described in the following manner.
\begin{theorem*}[\cite{will23}]
	For every $X \in \mathrm{dg\Lambda\!HopfMod}^{c}_{e_n^c}$ of configuration-space-type, with $H^{\sbullet}(X(1))$ connected and finite dimensional, there is a Maurer-Cartan element $Z \in \MC(\hgc_{\overline{H}^{\sbullet}(X(1))})$ and a weak equivalence
	\[
		\Graphs_{\overline{H}^{\sbullet}(X(1)), n}^Z \simeq X.
	\]
\end{theorem*}

Moreover, we give a description of the coinduction functor
\[
\coind_n^m\colon \mathrm{dg\Lambda\!HopfMod}^{c}_{e_m^c} \to \mathrm{dg\Lambda\!HopfMod}^{c}_{e_n^c}.
\]
restricted to the above essential image.

\begin{futuretheorem}{thm:coinduction graphs}
	If $n-m \geqslant 2$ and $m \geqslant 2$ then there is a weak equivalence of $\mathrm{dg\Lambda Hopf}e_n^c$-comodules
	\[
	\coind_n^m(\Graphs^{Z_{u}}_{U,m}) \simeq \Graphs^{Z_{u}^{\mathrm{t}}}_{U,n}.	
	\]
	Here $Z_u \in \hgc_{U, m}$ and $Z_u^{\mathrm t}$ is the tree part of $Z_u$ regraded to be an element in $\hgc_{U, n}$.
\end{futuretheorem}

The above description allows us to modify the composition map of \hr[Theorem]{thm:main-composition} to the case of 
a "map":
\[
	\map^h_{\mathrm{mod}-\mathcal F_m} (\mathcal F_M, \mathcal F_N^\Q)
	\times
	\map^h_{\mathrm{mod}-\mathcal F_n} (\mathcal F_N, \mathcal F_K^\Q)
	\dashrightarrow
	\map^h_{\mathrm{mod}-\mathcal F_m} (\mathcal F_M, \mathcal F_K^\Q)
\]
that naturally arises from the composition~\eqref{eq:emb-tilda composition}.

The appropriate modification happens in \hr[Section]{sec:general case}. The main result is the following theorem.

\begin{futuretheorem}{main-general}
	Let $\mc(\hgc_{V,W,n}^{Z_v, Z_w}) \to \mc(\hgc_{V,W,k}^{Z^{\mathrm{t}}_v, Z^{\mathrm{t}}_w})$
	be the map induced by the composite
	\[
	\hgc_{U,V,n} \to \hgc^{\mathrm{tree}}_{U,V,n} \to \hgc_{U,V,k}
	\]
	of projection onto the tree part and regrading of the tree.
	Then the composite
	\[
	\mc(\hgc_{V,W,k}^{Z_v, Z_w})\times \mc(\hgc_{U,V,n}^{Z_u, Z_v})
	\to
	\mc(\hgc_{V,W,k}^{Z_v, Z_w})\times \mc(\hgc_{U,V,k}^{Z^{\mathrm{t}}_v, Z^{\mathrm{t}}_w})
	\to
	\mc(\hgc_{U,W,k}^{Z_u,Z_v}),
	\]
	where the second arrow is given by \textup{\hr[Theorem]{thm:main-composition}}, is weakly equivalent to the "composition" map
	\[
	\begin{tikzcd}[column sep=0.5em,font=\small,ampersand replacement=\&]
		\map_n^h(\cores_k^n(\Graphs_{W,k}), \Graphs_{V,n})
		\arrow[r, phantom, "\times"]
		\&
		\map_m^h(\cores_n^m(\Graphs_{V,n}), \Graphs_{U,m})
		\arrow[rr, dashed]
		\&\&
		\map_m^h(\cores_k^m(\Graphs_{W,k}), \Graphs_{U,m})
	\end{tikzcd}
	\]
	that is defined via $\coind-\cores$ adjunction \textup{(\textit{see}~\hr[Section]{sec:general case})}.%
\end{futuretheorem}

\subsection{Conventions} In what follows the base field will always be the field of rational numbers~$\Q$.

Let $V^{\sbullet}$ be a graded vector space, $x_1, \ldots, x_n$ homogeneous elements of degrees $|x_1|, \ldots, |x_n|$, respectively, and $\sigma \in S_n$ a permutation. The Koszul sign $\epsilon(\sigma)$ is defined by the relation
\[
\epsilon(\sigma)x_{\sigma(1)}\cdots x_{\sigma(n)} = x_1 \cdots x_n
\]
in the graded symmetric algebra $S(V)$.

We (ab)use the notation $\map_n$ and $\mor_n$ for mapping spaces and morphism sets in the categories of $\mathrm{dg\Lambda\!Hopf}$ $e_n$-modules and $\mathrm{dg\Lambda\!Hopf}$ $e_n^c$-comodules.

\subsection{Acknowledgements}
I am deeply grateful to my supervisor, Thomas Willwacher, for posing the problem and for his patient guidance over many hours.

\section{$L_\infty$ preliminaries}

\subsection{$L_\infty$-algebras} A (flat) $L_\infty$-algebra (over $\Q$) is a graded vector space $L$ over $\Q$ endowed with a degree one codifferential $l$ on the reduced symmetric coalgebra $\big(S^+\big(L[1]\big), \overline\Delta_{sh}\big)$. 

An $L_\infty$-morphism between two $L_\infty$-algebras $F\colon (L, l) \to (L', l')$ is a morphism of graded codifferential coalgebras
\[
    F\colon S^+\big(L[1]\big) \to S^+\big(L'[1]\big),
\]
i.e. $Fl = l' F$.

Looking at Taylor coefficients of the coderivation $l$ we get the usual structure maps
\[
    l_n\colon S^n\big(L[1]\big) \to L[2],
\]
with the coderivation given by
\[
    l(x_1, \dots, x_n) = \sum_{k=1}^{n} \sum_{\sigma \in Sh(k,n-k)} \epsilon(\sigma) l_k(x_{\sigma(1)}, \dots, x_{\sigma(k)}) x_{\sigma(k+1)} \dots x_{\sigma(n)}
\]
for homogeneous $x_1, \ldots, x_n \in L$. The relation $l^2 = 0$ is equivalent to
\[
    \sum_{k=1}^{n} \sum_{\sigma \in Sh(k,n-k)} \epsilon(\sigma) l_{n-k+1} ( l_k ( x_{\sigma(1)}, \dots, x_{\sigma(k)} ), x_{\sigma(k+1)}, \dots, x_{\sigma(n)} ) = 0
\]

\subsection{$L_\infty$-derivations and actions}
Define a dg Lie algebra $\Der_{L_\infty} (L)$ of $L_\infty$-derivations of $L$ to be coderivations of the cofree cocommutative coalgebra $(S^+\big(L[1]\big), \overline\Delta_{sh})$:
\[
	\Der_{L_\infty} (L) \coloneqq \big\{ \delta \colon S^+\big(L[1]\big) \to S^+\big(L[1]\big)\,|\; \bar\Delta_{sh}\delta = (\delta \otimes \id + \id \otimes \delta) \bar\Delta_{sh}\big\},
\]
with the bracket
\[
    [\delta_1, \delta_2] \coloneqq \delta_1 \circ \delta_2 - (-1)^{|\delta_1| |\delta_2|} \delta_2 \circ \delta_1,
\]
where the degree $|\delta_i|$ is the degree of the map $\delta_i \colon S^+\big(L[1]\big) \to S^+\big(L[1]\big)$,
and the differential 
\[
    d\cdot \delta \coloneqq [l, \delta].
\]

Note that the Maurer-Cartan set
\[
    \MC\big(\Der_{L_\infty} (L)\big) \coloneqq \{\delta \in \Der_{L_\infty} (L)\,|\; d\cdot\delta + \textstyle{\frac 12}[\delta, \delta] = 0\}
\]
is exactly the set of deformations of the $L_\infty$-structure on $L$. Namely, $l + \delta$ defines an $L_\infty$-structure on $L$ for $\delta \in \Der_{L_\infty} (L)$ if and only if $\delta$ is a Maurer-Cartan element.

Let $\mathfrak g$ be a dg Lie algebra. A \emph{left} (resp. \emph{right}) $L_\infty$-\emph{action} of $\mathfrak g$ on an $L_\infty$-algebra $L$ is a dg Lie algebra homomorphism $\mathfrak g \to \Der_{L_\infty} (L)$ (resp. $\mathfrak g^{op} \to \Der_{L_\infty} (L)$).

Note that an(y) $L_\infty$-action of $\mathfrak g$ on $L$ induces a map of Maurer-Cartan sets
\[
    \MC(\mathfrak g) \to \MC\big(\Der_{L_\infty} (L)\big).
\]
Therefore, any Maurer-Cartan element in $\mathfrak g$ can be used to deform the $L_\infty$-structure on $L$. Namely, let the $L_\infty$-action of $\mathfrak g$ on $L$ be given by a morphism $\varphi\colon \mathfrak{g} \to \Der_{L_\infty} (L)$, and let $z \in \MC(\mathfrak{g})$ be a Maurer-Cartan element. We then define $L^z$ to be the $L_\infty$-algebra with the same underlying space and the twisted coderivation $l+\varphi(z)$.

Similarly, if we are given a pair (left and right) of actions $\varphi_1\colon \mathfrak g_1 \to \Der_{L_\infty} (L)$ and $\varphi_2\colon \mathfrak g_2^{op} \to \Der_{L_\infty} (L)$ and a couple of Maurer-Cartan elements $z_1 \in \MC(\mathfrak g_1)$ and $z_2 \in \MC(\mathfrak g_2)$, we can define $L^{z_1, z_2}$ to be the deformation of $L_\infty$-algebra structure with the coderivation given by $l + \varphi_1(z_1) + \varphi_2(z_2)$.

\vspace{3mm}
\begin{proposition}\label{prop:diagonal twist}
Let $\mathfrak g$ be a dg Lie algebra, and $L_1$, $L_2$, $L_3$ three $L_\infty$-algebras. Suppose that $L_1$, $L_2$ are endowed with right and left $L_\infty$-actions $\varphi_1\colon \mathfrak g^{op} \to \Der_{L_\infty} (L_1)$, $\varphi_2\colon \mathfrak g \to \Der_{L_\infty} (L_2)$ of $\mathfrak g$ respectively. Let
\[
    F\colon L_1\oplus L_2 \to L_3
\]
be a morphism of $L_\infty$-algebras that is invariant under the diagonal action of $\mathfrak g$, i.e. for $z \in \mathfrak g$, $x_1, \ldots, x_p \in L_1$, $y_1, \ldots, y_q \in L_2$ the equality holds
\begin{equation}\label{l-infty-invariance}
    F\big(\varphi_1(z)(x_1\cdots x_py_1\cdots y_q)\big)
    =
    F\big(\varphi_2(z)(x_1\cdots x_py_1\cdots y_q)\big).
\end{equation}
In the above, $\varphi_1$, $\varphi_2$ denote the actions on $L_1$ and $L_2$ respectively. Then for any $z \in \MC(\mathfrak g)$ there is an induced morphism
\[
    L_1^z \oplus L_2^z \to L_3
\]
from the direct sum of twisted $L_\infty$-algebras to $L_3$.
\end{proposition}
\begin{proof}
Denote by $l^\oplus$, $l^z$ the $L_\infty$-structures on the direct sum $L_1 \oplus L_2$ and on the twisted direct sum $L_1^z\oplus L_2^z$. Then we have
\[
    l^z(x_1\cdots x_p y_1 \cdots y_q) = l^\oplus(x_1\cdots x_p y_1 \cdots y_q) + \varphi_1(z)(x_1\cdots x_p y_1 \cdots y_q) - \varphi_2(z)(x_1\cdots x_p y_1 \cdots y_q).
\]
Therefore, applying $F$ to both sides and using the invariance~\eqref{l-infty-invariance} we get
\[
    F\big(l^z(x_1\cdots x_p y_1 \cdots y_q)\big) = F\big(l^\oplus(x_1\cdots x_p y_1 \cdots y_q)\big).
\]
Thus, $F$ is well-defined on the twisted direct sum.
\end{proof}

\begin{corollary}\label{l-infty-twist-composition}
Let $\mathfrak g_1$, $\mathfrak g_2$, $\mathfrak g_3$ be three dg Lie algebras, and $L_1$, $L_2$, $L_3$ three $L_\infty$-algebras. Suppose that
\begin{itemize}[leftmargin=0.03\textwidth]
\item $L_1$, $L_3$ have a left $L_\infty$-action of $\mathfrak g_1$;
\item $L_1$, $L_2$ have respectively right and left $L_\infty$-actions of $\mathfrak g_2$;
\item $L_2$, $L_3$ have a right $L_\infty$-action of $\mathfrak g_3$;
\end{itemize}
such that the given left and right actions on $L_i$, $i = 1,2,3$, commute.
Let $F\colon L_1\oplus L_2 \to L_3$
be a morphism of $L_\infty$-algebras that is $\mathfrak g_1$-$\mathfrak g_3$-equivariant and invariant under the diagonal action of $\mathfrak g_2$ \textup(see~\eqref{l-infty-invariance}\textup). Then for any triple $z_i \in \MC(\mathfrak g_i)$, $i = 1, 2, 3$ of Maurer-Cartan elements there is an induced morphism
\[
    L_1^{z_1, z_2} \oplus L_2^{z_2, z_3} \to L_3^{z_1, z_3}
\]
of twisted $L_\infty$-algebras, and, therefore, a morphism of Maurer-Cartan sets
\[
    \MC(L_1^{z_1, z_2}) \times \MC(L_2^{z_2, z_3}) \to \MC(L_3^{z_1, z_3}).
\]
\end{corollary}
\begin{proof}
	By \hr[Proposition]{prop:diagonal twist}, we get a morphism of $L_\infty$-algebras
	\[
		L_1^{z_2}\oplus L_2^{z_2} \to L_3.
	\]
	As the left and right actions commute, we can furthermore twist the given morphism with $z_1$ and $z_3$ to get the desired morphism
	\[
		L_1^{z_1, z_2} \oplus L_2^{z_2, z_3} \to L_3^{z_1, z_3}
	\]
	of twisted $L_\infty$-algebras.
\end{proof}

\section{Models for composition}

\begin{proposition}\label{sset-composition-model}
	Let $\mathcal C$ be a simplicial model category, and $A, B, C \in \mathcal C$ be cofibrant objects. Let $X_{\sbullet}, Y_{\sbullet}, Z_{\sbullet}$ be simplicial sets equipped with maps $X_{\sbullet} \to \map_{\mathcal C} (A, B)$, $Y_{\sbullet} \to \map_{\mathcal C} (B, C)$, $Z_{\sbullet} \to \map_{\mathcal C} (A, C)$ and $X_{\sbullet} \times Y_{\sbullet} \to Z_{\sbullet}$ making the following diagram commute
	\[
	\begin{tikzcd}
		X_{\sbullet} \arrow[r, symbol=\times] \arrow[d] & Y_{\sbullet} \arrow[r] \arrow[d]  & Z_{\sbullet} \arrow[d]
		\\
		\map_{\mathcal C} (A, B) \arrow[r, symbol=\times] & \map_{\mathcal C} (B, C) \arrow[r] & \map_{\mathcal C} (A, C).
	\end{tikzcd}
	\]
	Let $A \xrightarrow{q_A} \widehat{A}$, $B \xrightarrow{q_B} \widehat{B}$ and $C \xrightarrow{q_C} \widehat{C}$ be fibrant-cofibrant replacements that are furthermore cofibrations. Finally, assume that the composition maps
	\begin{alignat}{3}\label{eq1}
		X_{\sbullet} \to &\map_{\mathcal C} (A, B) \xrightarrow{(q_A)_*} &\map_{\mathcal C} (A, \widehat B) = &\map^h_{\mathcal C} (A, B)%
		\notag
		\\
		Y_{\sbullet} \to &\map_{\mathcal C} (B, C) \xrightarrow{(q_B)_*} &\map_{\mathcal C} (B, \widehat C) = &\map^h_{\mathcal C} (B, C)
		\\
		Z_{\sbullet} \to &\map_{\mathcal C} (A, C) \xrightarrow{(q_C)_*} &\map_{\mathcal C} (A, \widehat C) = &\map^h_{\mathcal C} (A, C)%
		\notag
	\end{alignat}
	are weak equivalences of simplicial sets.
	Then the map $X_{\sbullet} \times Y_{\sbullet} \to Z_{\sbullet}$ is weakly equivalent to the derived composition map $\map^h_{\mathcal C} (A, B)\times \map^h_{\mathcal C} (B, C) \to \map^h_{\mathcal C} (A, C)$.
\end{proposition}

\begin{proof}
	Consider the following diagram
	\begin{equation}\label{d1}
		\begin{tikzcd}
			\map_{\mathcal{C}}(A,\widehat B)
			\arrow[r, symbol=\times]
			\arrow[dr, swap, "\id", "\sim"']
			&
			\map_{\mathcal{C}}(\widehat B,\widehat C)
			\arrow[r]
			\arrow[dr, swap, "(q_B)^*", "\sim"']
			&
			\map_{\mathcal{C}}(A,\widehat C)
			\arrow[dr, swap, "\id", "\sim"']
			\\
			&
			\map_{\mathcal{C}}(A,\widehat B)
			\arrow[r, symbol=\times]
			&
			\map_{\mathcal{C}}(B,\widehat C)
			&
			\map_{\mathcal{C}}(A,\widehat C)
			\\
			X_{\sbullet}
			\arrow[r, symbol=\times]
			\arrow[ur, "\sim"]
			&
			Y_{\sbullet}
			\arrow[r]
			\arrow[ur, "\sim"]
			&
			Z_{\sbullet}
			\arrow[ur, "\sim"]
		\end{tikzcd}
	\end{equation}
	where the arrows bottom-to-top are given by the compositions~\eqref{eq1}.
	The top horizontal arrow is given by the composition map and represents the derived composition map. Note that all diagonal arrows are weak equivalences.
	
	Let $F_{\sbullet}, G_{\sbullet}, H_{\sbullet}$ be the pullbacks of the following diagrams
	\begin{center}
		\begin{tikzcd}[row sep=.2em]
			X_{\sbullet}
			\arrow[r, "\sim"]
			&
			\map_{\mathcal{C}}(A,\widehat B)
			&
			\map_{\mathcal{C}}(\widehat B,\widehat C);
			\arrow[l, "\id"']
			\\
			Y_{\sbullet}
			\arrow[r, "\sim"]
			&
			\map_{\mathcal{C}}(B,\widehat C)
			&
			\map_{\mathcal{C}}(\widehat B,\widehat C);
			\arrow[l, "(q_B)^*"']
			\\
			Z_{\sbullet}
			\arrow[r, "\sim"]
			&
			\map_{\mathcal{C}}(A,\widehat C)
			&
			\map_{\mathcal{C}}(A,\widehat C).
			\arrow[l, "\id"']
		\end{tikzcd}
	\end{center}
	By the universal property, the composites
	\[
	F_{\sbullet}\times G_{\sbullet} \to \map_{\mathcal{C}}(A,\widehat B) \times \map_{\mathcal{C}}(\widehat B, \widehat C) \to \map_{\mathcal{C}}(A,\widehat C)
	\]
	and
	\[
	F_{\sbullet}\times G_{\sbullet} \to X_{\sbullet}\times Y_{\sbullet} \to Z_{\sbullet}
	\]
	define a map
	\[
	F_{\sbullet}\times G_{\sbullet} \to H_{\sbullet}
	\]
	that fits into the following commutative diagram
	\[
	\begin{tikzcd}
		&\map_{\mathcal{C}}(A,\widehat B)
		\arrow[r, symbol=\times]
		&
		\map_{\mathcal{C}}(\widehat B,\widehat C)
		\arrow[r]
		&
		\map_{\mathcal{C}}(A,\widehat C)
		\\
		F_{\sbullet}
		\arrow[r, symbol=\times]
		\arrow[ur]
		\arrow[dr]
		&
		G_{\sbullet}
		\arrow[r]
		\arrow[ur]
		\arrow[dr]
		&
		H_{\sbullet}
		\arrow[ur]
		\arrow[dr]
		\\
		&
		X_{\sbullet}
		\arrow[r, symbol=\times]
		&
		Y_{\sbullet}
		\arrow[r]
		&
		Z_{\sbullet}
	\end{tikzcd}
	\]
	The diagonal arrows are weak equivalences as pullbacks of diagonal arrows of~\eqref{d1}
\end{proof}

\section{Graphical categories}

\subsection{Graph complexes} In this section we give an overview of graph and hairy graph complexes. We refer to \cite{will23} for the comprehensive exposition. In what is coming we will use both normal and dualized versions of the graph complex. 

We start with the dg$\Lambda$Hopf $e_n^c$-comodule $\Graphs_{U,n}^\#$, $n \geqslant 2$. Let $U$ be a graded vector space. Define $\Graphs_{U,n}^\#(r)$ to be a vector space spanned by isomorphism classes of oriented graphs with $r$ numbered "\emph{external}" vertices and an arbitrary finite number of unnumbered "\emph{internal}" vertices (black in the pictures), with no connected components without an external vertex. Moreover, all vertices can be decorated by elements of the symmetric algebra $S(U)$.
\[
\begin{tikzpicture}
	\node[ext] (v1) at (0,0) {$\scriptstyle 1$};
	\node[ext] (v2) at (.7,0) {$\scriptstyle 2$};
	\node[ext,label=0:{$\gamma$}] (v3) at (1.4,0) {$\scriptstyle 3$};
	\node[int,label=90:{$\alpha\beta$}] (i1) at (.35,.7) {};
	\node[int] (i2) at (1.05,.7) {};
	\draw (v1) edge (i1) 
	(i1) edge (i2) edge (v2) 
	(i2) edge (v2) edge (v3)
	(v2) edge[bend left] (v3);
\end{tikzpicture}
\in \Graphs_{U,n}^\#(3), \text{with $\alpha,\beta,\gamma\in U$.} 
\]
Throughout the paper we use cohomological degree conventions. Namely, the internal vertices have degree $-n$, the edges have degree $n-1$:
\[
	\deg(\Gamma) = (n-1)\#\{\text{edges of $\Gamma$}\} - n\#\{\text{internal vertices of $\Gamma$}\}.
\]
The orientation converts into the following conventions for the graphs depending on the parity of $n$:
\begin{itemize}
\item[even $n$:] Graphs have an order on the set of edges. Changing the order by a permutation $\sigma$ alters the sign of the graph by $\sgn(\sigma)$.

\item[odd $n$:] Graphs have an order on the set of internal vertices. Changing the order by a permutation $\sigma$ changes the sign of the graph by $\sgn(\sigma)$. Moreover, the edges have orientations and the different orientations are identified up to sign.
\[
	\longrightarrow {}={} - \longleftarrow.
\]
\end{itemize}
The multiplication on $\Graphs_{U,n}^\#(r)$ is given by gluing graphs along external vertices.
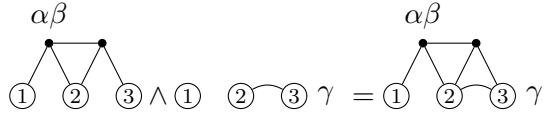
\[
\begin{tikzpicture}
	\node[ext] (v1) at (0,0) {$\scriptstyle 1$};
	\node[ext] (v2) at (.7,0) {$\scriptstyle 2$};
	\node[ext] (v3) at (1.4,0) {$\scriptstyle 3$};
	\node[int,label=90:{$\alpha\beta$}] (i1) at (.35,.7) {};
	\node[int] (i2) at (1.05,.7) {};
	\draw (v1) edge (i1) 
	(i1) edge (i2) edge (v2) 
	(i2) edge (v2) edge (v3);
\end{tikzpicture}
\wedge
\begin{tikzpicture}
	\node[ext] (v1) at (0,0) {$\scriptstyle 1$};
	\node[ext] (v2) at (.7,0) {$\scriptstyle 2$};
	\node[ext,label=0:{$\gamma$}] (v3) at (1.4,0) {$\scriptstyle 3$};
	\draw (v2) edge[bend left] (v3);
\end{tikzpicture}
=
\begin{tikzpicture}
	\node[ext] (v1) at (0,0) {$\scriptstyle 1$};
	\node[ext] (v2) at (.7,0) {$\scriptstyle 2$};
	\node[ext,label=0:{$\gamma$}] (v3) at (1.4,0) {$\scriptstyle 3$};
	\node[int,label=90:{$\alpha\beta$}] (i1) at (.35,.7) {};
	\node[int] (i2) at (1.05,.7) {};
	\draw (v1) edge (i1) 
	(i1) edge (i2) edge (v2) 
	(i2) edge (v2) edge (v3)
	(v2) edge[bend left] (v3);
\end{tikzpicture}
\]
The differential is given by contraction of the edges between two distinct internal vertices.\footnote{%
If the order of edges and vertices is such that $e$ is the first edge, connecting the first internal vertex in the ordering to the second, then the sign is ”$+$”, with the ordering of the remaining edges and vertices in $\Gamma/e$ being the same as in $\Gamma$.}
The decoration of the resulting vertex is the product of the decorations at the ends of the edge.
At the moment we don't put any extra requirements on the valency, but for the future notice: each vertex decoration by $U$ contributes~$1$ to the valency of the vertex.

The collection $\{\Graphs_{U,n}^\#(r)\}_{r \in \Z_{>0}}$ comes with the structure of a right $e_n^c$-comodule. The right action of the symmetric group $S_r$ on $\Graphs_{U,n}^\#(r)$ is defined by permuting the numbering of the external vertices. The right $e_n^c$-coaction is given by contracting subgraphs with no internal vertices. Namely, let
\[
	\Delta_{s}\colon \Graphs_{U,n}^\#(r) \to \Graphs_{U,n}^\#(r-s+1) \otimes e_n^c(s)
\]
be the cooperadic coactions corresponding to the subset $\{1, \ldots, s\} \subseteq \{1, \ldots, r\}$. Then, for a graph $\Gamma \in \Graphs_{U,n}^\#(r)$,
\[
	\Delta_s(\Gamma) = \sum \pm (\Gamma/\gamma) \otimes \gamma,
\]
where the sum runs over all subgraphs $\gamma \subseteq \Gamma$ that contain the external vertices $1, \ldots, s$, and no other vertices, and $\Gamma/\gamma$ the graph obtained by contracting $\gamma$ to one new external vertex numbered $1$ and renumbering the external vertices $s+1, \ldots, r$ to $2, \ldots, r-s+1$.
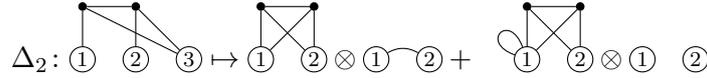
\[
\Delta_2\colon
\begin{tikzpicture}
	\node[ext] (v1) at (0,0) {$\scriptstyle 1$};
	\node[ext] (v2) at (.7,0) {$\scriptstyle 2$};
	\node[ext] (v3) at (1.4,0) {$\scriptstyle 3$};
	\node[int] (i1) at (0,.7) {};
	\node[int] (i2) at (.7,.7) {};
	\draw (v1) edge (i1)
		  (v2) edge (i2)
		  (v3) edge (i1) edge (i2)
		  (i1) edge (i2);
\end{tikzpicture}
\mapsto
\begin{tikzpicture}
	\node[ext] (v1) at (0,0) {$\scriptstyle 1$};
	\node[ext] (v2) at (.7,0) {$\scriptstyle 2$};
	\node[int] (i1) at (0,.7) {};
	\node[int] (i2) at (.7,.7) {};
	\draw (v1) edge (i1) edge (i2)
	(v2) edge (i1) edge (i2)
	(i1) edge (i2);
\end{tikzpicture}
\otimes
\begin{tikzpicture}
	\node[ext] (v1) at (0,0) {$\scriptstyle 1$};
	\node[ext] (v2) at (.7,0) {$\scriptstyle 2$};
	\draw (v1) edge[bend left] (v2);
\end{tikzpicture}
+
\begin{tikzpicture}
	\node[ext] (v1) at (0,0) {$\scriptstyle 1$};
	\node[ext] (v2) at (.7,0) {$\scriptstyle 2$};
	\node[int] (i1) at (0,.7) {};
	\node[int] (i2) at (.7,.7) {};
	\draw (v1) edge (i1) edge (i2)
	(v2) edge (i1) edge (i2)
	(i1) edge (i2);
	\draw (v1) edge[in=110,out=160,loop] (v1);
\end{tikzpicture}
\otimes
\begin{tikzpicture}
	\node[ext] (v1) at (0,0) {$\scriptstyle 1$};
	\node[ext] (v2) at (.7,0) {$\scriptstyle 2$};
\end{tikzpicture}
\]

\subsection{Hairy graph complexes}
Now we pass to a dg Lie algebra $\hgc^\#_{U, n}$. As a vector space $\hgc^\#_{U,n}$ is spanned by isomorphism classes of connected graphs (as before) with all external vertices ("\emph{hairs}") having valency exactly~$1$ and a decoration by $U_1 \coloneqq \Q\cdot 1 \oplus U$, where $1$ is a formal symbol of degree $0$:
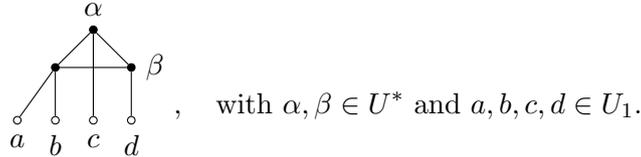
\[
\begin{tikzpicture}[yshift=-.5cm]
\node[int] (i1) at (.5,1) {};
\node[int,label=90:{$\alpha$}] (i2) at (1,1.5) {};
\node[int,label=0:{$\beta$}] (i3) at (1.5,1) {};
\node[ext,label=-90:{$a$}] (e1) at (0,.3) {};
\node[ext,label=-90:{$b$}] (e2) at (.5,.3) {};
\node[ext,label=-90:{$c$}] (e3) at (1,.3) {};
\node[ext,label=-90:{$d$}] (e4) at (1.5,.3) {};
\draw (i1) edge (i2) edge (i3) edge (e1) edge (e2)
(i2) edge (i3) edge (e3)
(i3) edge (e4);
\end{tikzpicture},
\quad
\text{with $\alpha, \beta \in U^*$ and $a,b,c,d \in U_1$.}
\]
The differential is given by the operation that is dual to the one described above. Namely, the differential is given by splitting a vertex in all possible ways.

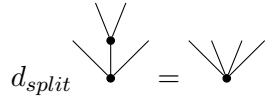
\begin{align*}
	d_{split}
	\begin{tikzpicture}[baseline=-.65ex]
		\node[int] (v) at (0,0) {};
		\node[int](w) at (0,.5) {};
		\draw (v) edge +(-.5,.5) edge +(.5,.5) edge (w) (w) edge +(-.2,.5) edge +(.2,.5);
	\end{tikzpicture}
	&=
	\begin{tikzpicture}[baseline=-.65ex]
		\node[int] (v) {};
		\draw (v) edge +(-.5,.5) edge +(-.2,.5) edge +(.2,.5) edge +(.5,.5);
	\end{tikzpicture}
\end{align*}

The Lie bracket is given by a commutator $[\Gamma_1,\Gamma_2]=\Gamma_1*\Gamma_2- (-1)^{|\Gamma_1||\Gamma_2|}\Gamma_2*\Gamma_1$, where
the Lie-admissible product $\Gamma_1 * \Gamma_2$ of two graphs is obtained by summing over all partial matchings of hairs in $\Gamma_1$ to decorations in $\Gamma_2$, gluing the hair to the decorated vertex and multiplying by the natural pairing of the decorations. The decoration $1$ at hairs is treated specially: these hairs of $\Gamma_1$ can be connected to any vertex of $\Gamma_2$:
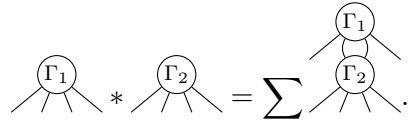
\[
\begin{tikzpicture}
	\node[ext] (v) at (0,.3) {$\scriptstyle \Gamma_1$};
	\draw (v) edge +(-.6,-.5) edge +(-.2,-.5) edge +(.2,-.5) edge +(.6,-.5) ;
\end{tikzpicture}
*
\begin{tikzpicture}
	\node[ext] (v) at (0,.3) {$\scriptstyle \Gamma_2$};
	\draw (v) edge +(-.6,-.5) edge +(-.2,-.5) edge +(.2,-.5) edge +(.6,-.5) ;
\end{tikzpicture}
=
\sum 
\begin{tikzpicture}
	\node[ext] (v1) at (0,.3) {$\scriptstyle \Gamma_2$};
	\node[ext] (v2) at (0,1) {$\scriptstyle \Gamma_1$};
	\draw (v2) edge[bend left] (v1) edge[bend right] (v1) edge +(-.6,-.5)  edge +(.6,-.5) 
	(v1) edge +(-.6,-.5) edge +(-.2,-.5) edge +(.2,-.5) edge +(.6,-.5) ;
\end{tikzpicture}.
\]

Recall that $\Graphs^{\#\,*}_{U,n}$ admits an action of $\hgc^\#_{U, n}$ (see~\cite[Section 8.3]{will23}). The action is of $\Gamma_u \in \hgc^\#_{U, n}$ on $\gamma_u \in \Graphs^{\#\,*}_{U,n}$ is given by:
\[
	\Gamma_u \cdot \gamma_u \coloneqq \sum\pm
	\begin{tikzpicture}[baseline=-.8ex]
		\node[draw,circle] (v0) at (0,1.3) {$\Gamma_u$};
		\node[draw,circle] (v) at (0,0) {$\gamma_u$};
		\node[ext] (w1) at (-.7,-1) {};
		\node[ext] (w2) at (-.35,-1) {};
		\node[label=center:{$\scriptstyle\cdots$}] at (0,-1) {};
		\node[ext] (w4) at (.35,-1) {};
		\node[ext] (w5) at (.7,-1) {};
		\draw (v0) edge[bend right] (v) edge[bend right=60] (v) edge[bend left] (v) edge[bend left=60] (v) edge[bend left=47] (w5);
		\draw (v) edge[bend right] (w1) edge (w2) edge[bend left] (w2) edge (w4) edge (w5) edge[bend left] (w5);
	\end{tikzpicture},
\]
where the sum runs over all possible ways to attach hairs of $\Gamma_u$ to $\gamma_u$.

Finally, we define an $L_\infty$-algebra $\hgc^\#_{U, V, n}$. As a vector space $\hgc^\#_{U, V, n}$ is spanned by isomorphism classes of connected graphs similar to $\hgc^\#_{U, n}$ with the difference that now decorations of the internal vertices are given by a graded vector space $V^*$. At the moment, we endow $\hgc^\#_{U,V,n}$ with an abelian $L_\infty$-structure.

\subsection{Graphical actions}\label{graph-action}
In this section we endow $\hgc^\#_{U, V, n}$ with a left $L_\infty$-action of $\hgc^\#_{V, n}$ (resp. right action of $\hgc^\#_{U, n}$). 

As any coderivation is uniquely defined by its Taylor coefficients, the homomorphism
\begin{equation}\label{left-action}
    \hgc^\#_{V,n} \to \Der_{L_\infty} (\hgc^\#_{U, V, n})
\end{equation}
is uniquely defined by the maps 
\[
\begin{tikzpicture}[baseline=-.8ex]
\node[draw,circle] (v) at (0,.3) {$\Gamma_v$};
\node[ext] (w1) at (-.7,-.5) {};
\node[ext] (w2) at (-.35,-.5) {};
\node[label=center:{$\scriptstyle\cdots$}] (w3) at (0,-.5) {};
\node[ext] (w4) at (.35,-.5) {};
\node[ext] (w5) at (.7,-.5) {};
\draw (v) edge[bend right] (w1) edge (w2) edge (w4) edge[bend left] (w5);
\end{tikzpicture}\cdot%
\colon
S^k(\hgc^\#_{U, V, n}[1]) \to \hgc^\#_{U, V, n}[2].
\]
Define the above in the following way
\[
\begin{tikzpicture}[baseline=-.8ex]
\node[draw,circle] (v) at (0,.3) {$\Gamma_v$};
\node[ext] (w1) at (-.7,-.5) {};
\node[ext] (w2) at (-.35,-.5) {};
\node[label=center:{$\scriptstyle\cdots$}] (w3) at (0,-.5) {};
\node[ext] (w4) at (.35,-.5) {};
\node[ext] (w5) at (.7,-.5) {};
\draw (v) edge[bend right] (w1) edge (w2) edge (w4) edge[bend left] (w5);
\end{tikzpicture}\cdot%
\Big(
\begin{tikzpicture}[baseline=-.8ex]
\node[draw,circle] (v) at (0,.3) {$\Gamma_{x_1}$};
\node[ext] (w1) at (-.7,-.5) {};
\node[ext] (w2) at (-.35,-.5) {};
\node[label=center:{$\scriptstyle\cdots$}] (w3) at (0,-.5) {};
\node[ext] (w4) at (.35,-.5) {};
\node[ext] (w5) at (.7,-.5) {};
\draw (v) edge[bend right] (w1) edge (w2) edge (w4) edge[bend left] (w5);
\end{tikzpicture};
\cdots;
\begin{tikzpicture}[baseline=-.8ex]
\node[draw,circle] (v) at (0,.3) {$\Gamma_{x_k}$};
\node[ext] (w1) at (-.7,-.5) {};
\node[ext] (w2) at (-.35,-.5) {};
\node[label=center:{$\scriptstyle\cdots$}] (w3) at (0,-.5) {};
\node[ext] (w4) at (.35,-.5) {};
\node[ext] (w5) at (.7,-.5) {};
\draw (v) edge[bend right] (w1) edge (w2) edge (w4) edge[bend left] (w5);
\end{tikzpicture}
\Big)
\coloneqq
\sum\pm
\begin{tikzpicture}[baseline=-.8ex]
\node[draw,circle] (v) at (0, 1.5) {$\scriptstyle{\Gamma_v}$};
\node[draw,circle] (v1) at (-1.2,.3) {$\scriptstyle{\Gamma_{x_1}}$};
\node[label=center:{$\cdots$}] at (0,.3) {};
\node[draw,circle] (v3) at (1.2,.3) {$\scriptstyle{\Gamma_{x_k}}$};
\node[ext,label=270:{\scalebox{0.5}{$1$}}] (w1) at (2.75,-1) {};
\node[label=center:{$\scriptstyle\cdots$}] (w2) at (3.1,-1) {};
\node[ext,label=270:{\scalebox{0.5}{$1$}}] (w3) at (3.45,-1) {};
\node[ext] (w11) at (-2.05,-1) {};
\node[ext] (w12) at (-1.7,-1) {};
\node[label=center:{$\scriptstyle\cdots$}] at (-1.2,-1) {};
\node[ext] (w13) at (-.8, -1) {};
\node[ext] (w14) at (-.45,-1) {};
\node[ext] (w21) at (0.35,-1) {};
\node[ext] (w22) at (0.7,-1) {};
\node[label=center:{$\scriptstyle\cdots$}] at (1.2,-1) {};
\node[ext] (w23) at (1.6,-1) {};
\node[ext] (w24) at (1.95,-1) {};
\draw (v) edge[bend right] (v1)
edge (v1)
edge[bend left] (v1)
edge[bend right] (v3)
edge (v3)
edge[bend left] (v3)
edge[bend left=50] (w1)
edge[bend left=60] (w3);
\draw (v1)
edge[bend right] (w11)
edge (w12)
edge (w13)
edge[bend left] (w14);
\draw (v3)
edge[bend right] (w21)
edge (w22)
edge (w23)
edge[bend left] (w24);
\end{tikzpicture},
\]
where the sum runs over all ways to attach $\Gamma_v$ to $\Gamma_{x_1}, \ldots, \Gamma_{x_k}$ in such way that
\begin{itemize}[leftmargin=0.04\textwidth]
\item the resulting graph is internally connected;
\item all $V$ decorated hairs of $\Gamma_v$  are attached;
\item hairs of $\Gamma_v$ decorated by~$1$ can remain unattached.
\end{itemize}

To check that~\eqref{left-action} is a homorphism of Lie algebras we need to check the following identity
\begin{align*}
\big[\Gamma_{v_1}, \Gamma_{v_2}\big]\cdot\big(\Gamma_{x_1}; \cdots; \Gamma_{x_k}\big)
=
\sum_{\sigma \in Sh(i, k-i)} \pm\Big(
\Gamma_{v_1}\cdot\big(
\Gamma_{v_2}\cdot(\Gamma_{x_{\sigma(1)}}; \cdots; \Gamma_{x_{\sigma(i)}}); \Gamma_{x_{\sigma(i+1)}}; \cdots; \Gamma_{x_{\sigma(k)}}
\big)
\\
\pm
\Gamma_{v_2}\cdot\big(
\Gamma_{v_1}\cdot(\Gamma_{x_{\sigma(1)}}; \cdots; \Gamma_{x_{\sigma(i)}}); \Gamma_{x_{\sigma(i+1)}}; \cdots; \Gamma_{x_{\sigma(k)}}
\big)\Big).
\end{align*}
To do so, we start with combinatorial description of the right-hand side, namely of the factor
\begin{equation}\label{rhside}
\Gamma_{v_1}\cdot\big(
\Gamma_{v_2}\cdot(\Gamma_{x_{\sigma(1)}}; \cdots; \Gamma_{x_{\sigma(i)}}); \Gamma_{x_{\sigma(i+1)}}; \cdots; \Gamma_{x_{\sigma(k)}}
\big).
\end{equation}

The resulting sum splits into two pieces: graphs with an edge from $\Gamma_{v_1}$ to $\Gamma_{v_2}$ (see~\eqref{rhside-connected}),
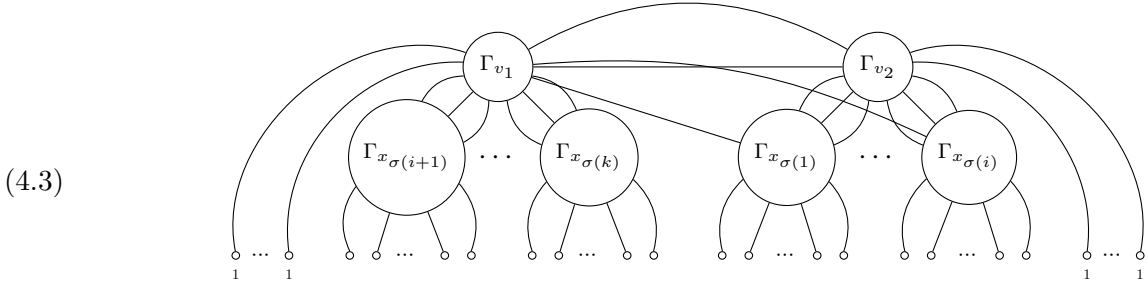
\begin{equation}\label{rhside-connected} 
	\begin{tikzpicture}[baseline=-.8ex]
		\node[draw,circle] (v1) at (-2.5, 1.5) {$\scriptstyle{\Gamma_{v_1}}$};
		\node[draw,circle] (v11) at (-1.3, .3) {$\scriptstyle{\Gamma_{x_{\sigma(k)}}}$};
		\node[label=center:{$\cdots$}] at (-2.5, .3) {};
		\node[draw,circle] (v31) at (-3.7, .3) {$\scriptstyle{\Gamma_{x_{\sigma(i+1)}}}$};
		\node[ext,label=270:{\scalebox{0.5}{$1$}}] (w11) at (-5.25, -1) {};
		\node[label=center:{$\scriptstyle\cdots$}] (w21) at (-5.6, -1) {};
		\node[ext,label=270:{\scalebox{0.5}{$1$}}] (w31) at (-5.95, -1) {};
		
		\node[ext] (w111) at (-0.45, -1) {};
		\node[ext] (w121) at (-0.8, -1) {};
		\node[label=center:{$\scriptstyle\cdots$}] at (-1.3, -1) {};
		\node[ext] (w131) at (-1.7, -1) {};
		\node[ext] (w141) at (-2.05, -1) {};
		
		\node[ext] (w211) at (-2.85, -1) {};
		\node[ext] (w221) at (-3.2, -1) {};
		\node[label=center:{$\scriptstyle\cdots$}] at (-3.7, -1) {};
		\node[ext] (w231) at (-4.1, -1) {};
		\node[ext] (w241) at (-4.45, -1) {};
		
		\draw (v1) edge[bend right] (v11)
		edge (v11)
		edge[bend left] (v11)
		edge[bend right] (v31)
		edge (v31)
		edge[bend left] (v31)
		edge[bend right=50] (w11)
		edge[bend right=60] (w31);
		
		\draw (v11) edge[bend left] (w111)
		edge (w121)
		edge (w131)
		edge[bend right] (w141);
		
		\draw (v31) edge[bend left] (w211)
		edge (w221)
		edge (w231)
		edge[bend right] (w241);
		
		\node[draw,circle] (v2) at (2.5, 1.5) {$\scriptstyle{\Gamma_{v_2}}$};
		\node[draw,circle] (v21) at (1.3, .3) {$\scriptstyle{\Gamma_{x_{\sigma(1)}}}$};
		\node[label=center:{$\cdots$}] at (2.5, .3) {};
		\node[draw,circle] (v23) at (3.7, .3) {$\scriptstyle{\Gamma_{x_{\sigma(i)}}}$};
		\node[ext,label=270:{\scalebox{0.5}{$1$}}] (w21) at (5.25, -1) {};
		\node[label=center:{$\scriptstyle\cdots$}] (w22) at (5.6, -1) {};
		\node[ext,label=270:{\scalebox{0.5}{$1$}}] (w23) at (5.95, -1) {};
		
		\node[ext] (w211) at (0.45, -1) {};
		\node[ext] (w212) at (0.8, -1) {};
		\node[label=center:{$\scriptstyle\cdots$}] at (1.3, -1) {};
		\node[ext] (w213) at (1.7, -1) {};
		\node[ext] (w214) at (2.05, -1) {};
		
		\node[ext] (w221) at (2.85, -1) {};
		\node[ext] (w222) at (3.2, -1) {};
		\node[label=center:{$\scriptstyle\cdots$}] at (3.7, -1) {};
		\node[ext] (w223) at (4.1, -1) {};
		\node[ext] (w224) at (4.45, -1) {};
		
		\draw (v2) edge[bend right] (v21)
		edge (v21)
		edge[bend left] (v21)
		edge[bend right] (v23)
		edge (v23)
		edge[bend left] (v23)
		edge[bend left=50] (w21)
		edge[bend left=60] (w23);
		
		\draw (v21) edge[bend right] (w211)
		edge (w212)
		edge (w213)
		edge[bend left] (w214);
		
		\draw (v23) edge[bend right] (w221)
		edge (w222)
		edge (w223)
		edge[bend left] (w224);
		
		\draw (v1) edge (v2)
		edge[bend left] (v2);
		\draw (v1) edge (v21)
		edge[bend left=15] (v23);
	\end{tikzpicture}
\end{equation}
and graphs with no edges between $\Gamma_{v_1}$ and $\Gamma_{v_2}$ (see~\eqref{rhside-disconnected}).

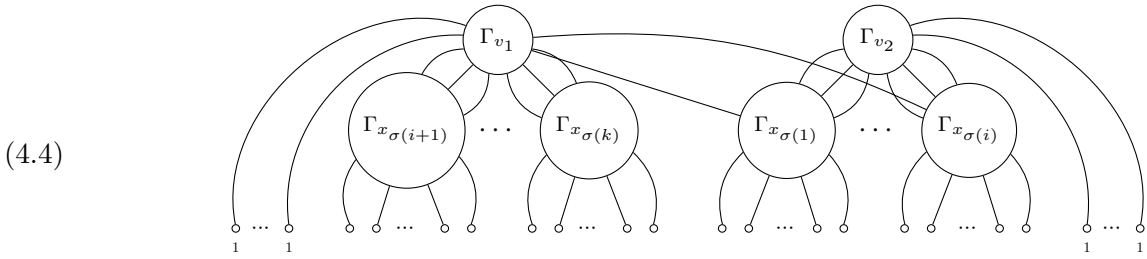
\begin{equation}\label{rhside-disconnected} 
	\begin{tikzpicture}[baseline=-.8ex]
		\node[draw,circle] (v1) at (-2.5, 1.5) {$\scriptstyle{\Gamma_{v_1}}$};
		\node[draw,circle] (v11) at (-1.3, .3) {$\scriptstyle{\Gamma_{x_{\sigma(k)}}}$};
		\node[label=center:{$\cdots$}] at (-2.5, .3) {};
		\node[draw,circle] (v31) at (-3.7, .3) {$\scriptstyle{\Gamma_{x_{\sigma(i+1)}}}$};
		\node[ext,label=270:{\scalebox{0.5}{$1$}}] (w11) at (-5.25, -1) {};
		\node[label=center:{$\scriptstyle\cdots$}] (w21) at (-5.6, -1) {};
		\node[ext,label=270:{\scalebox{0.5}{$1$}}] (w31) at (-5.95, -1) {};
		
		\node[ext] (w111) at (-0.45, -1) {};
		\node[ext] (w121) at (-0.8, -1) {};
		\node[label=center:{$\scriptstyle\cdots$}] at (-1.3, -1) {};
		\node[ext] (w131) at (-1.7, -1) {};
		\node[ext] (w141) at (-2.05, -1) {};
		
		\node[ext] (w211) at (-2.85, -1) {};
		\node[ext] (w221) at (-3.2, -1) {};
		\node[label=center:{$\scriptstyle\cdots$}] at (-3.7, -1) {};
		\node[ext] (w231) at (-4.1, -1) {};
		\node[ext] (w241) at (-4.45, -1) {};
		
		\draw (v1) edge[bend right] (v11)
		edge (v11)
		edge[bend left] (v11)
		edge[bend right] (v31)
		edge (v31)
		edge[bend left] (v31)
		edge[bend right=50] (w11)
		edge[bend right=60] (w31);
		
		\draw (v11) edge[bend left] (w111)
		edge (w121)
		edge (w131)
		edge[bend right] (w141);
		
		\draw (v31) edge[bend left] (w211)
		edge (w221)
		edge (w231)
		edge[bend right] (w241);
		
		\node[draw,circle] (v2) at (2.5, 1.5) {$\scriptstyle{\Gamma_{v_2}}$};
		\node[draw,circle] (v21) at (1.3, .3) {$\scriptstyle{\Gamma_{x_{\sigma(1)}}}$};
		\node[label=center:{$\cdots$}] at (2.5, .3) {};
		\node[draw,circle] (v23) at (3.7, .3) {$\scriptstyle{\Gamma_{x_{\sigma(i)}}}$};
		\node[ext,label=270:{\scalebox{0.5}{$1$}}] (w21) at (5.25, -1) {};
		\node[label=center:{$\scriptstyle\cdots$}] (w22) at (5.6, -1) {};
		\node[ext,label=270:{\scalebox{0.5}{$1$}}] (w23) at (5.95, -1) {};
		
		\node[ext] (w211) at (0.45, -1) {};
		\node[ext] (w212) at (0.8, -1) {};
		\node[label=center:{$\scriptstyle\cdots$}] at (1.3, -1) {};
		\node[ext] (w213) at (1.7, -1) {};
		\node[ext] (w214) at (2.05, -1) {};
		
		\node[ext] (w221) at (2.85, -1) {};
		\node[ext] (w222) at (3.2, -1) {};
		\node[label=center:{$\scriptstyle\cdots$}] at (3.7, -1) {};
		\node[ext] (w223) at (4.1, -1) {};
		\node[ext] (w224) at (4.45, -1) {};
		
		\draw (v2) edge[bend right] (v21)
		edge (v21)
		edge[bend left] (v21)
		edge[bend right] (v23)
		edge (v23)
		edge[bend left] (v23)
		edge[bend left=50] (w21)
		edge[bend left=60] (w23);
		
		\draw (v21) edge[bend right] (w211)
		edge (w212)
		edge (w213)
		edge[bend left] (w214);
		
		\draw (v23) edge[bend right] (w221)
		edge (w222)
		edge (w223)
		edge[bend left] (w224);
		
		\draw (v1) edge (v21)
		edge[bend left=15] (v23);
	\end{tikzpicture}
\end{equation}
Note that the latter one also appears as the corresponding part in
\[
\Gamma_{v_2}\cdot\big(
\Gamma_{v_1}\cdot(\Gamma_{x_{\sigma'(1)}}; \cdots; \Gamma_{x_{\sigma'(j)}}); \Gamma_{x_{\sigma'(j+1)}}; \cdots; \Gamma_{x_{\sigma'(k)}}
\big),
\]
where $\Gamma_{x_{\sigma'(1)}}, \ldots, \Gamma_{x_{\sigma'(j)}}$ are graphs to which $\Gamma_{v_1}$ was attached in~\eqref{rhside-disconnected}. Therefore, part of~\eqref{rhside} that remains, has only graphs of the form~\eqref{rhside-connected}, i.e. the graphs where at least one hair of $\Gamma_{v_1}$ is attached to $\Gamma_{v_2}$. The same holds for
\[
\Gamma_{v_2}\cdot\big(
\Gamma_{v_1}\cdot(\Gamma_{x_{\sigma(1)}}; \cdots; \Gamma_{x_{\sigma(i)}}); \Gamma_{x_{\sigma(i+1)}}; \cdots; \Gamma_{x_{\sigma(k)}}
\big).
\]
As the result we have that the only graphs contributing to the right-hand side are the ones with hairs of $\Gamma_{v_1}$ attached to $\Gamma_{v_2}$ or other way around, which exactly describes the left-hand side.

Finally, we need to check compatibility with the differentials.
Namely, we need to verify that
\begin{equation}\label{eq:dc0}
	(d_{split} \Gamma_v)\cdot = [d_{split}, \Gamma_v\cdot].
\end{equation}
Let us start by unfolding the right hand side
\begin{equation}\label{eq:dc1}
	d_{split}\circ (\Gamma_v\cdot) - (-1)^{|\Gamma_v|}(\Gamma_v\cdot)\circ d_{split}.
\end{equation}
The result of the first summand applied to $\Gamma_{x_1}, \ldots, \Gamma_{x_k}$ splits into two pieces:
\begin{align*}
	d_{split}\circ (\Gamma_v\cdot)(\Gamma_{x_1}; \cdots; \Gamma_{x_k})
	&=
	d_{split}\Biggl(\sum\pm
	\begin{tikzpicture}[baseline=-.8ex]
		\node[draw,circle] (v) at (0, 1.5) {$\scriptstyle{\Gamma_v}$};
		\node[draw,circle] (v1) at (-1.2,.3) {$\scriptstyle{\Gamma_{x_1}}$};
		\node[label=center:{$\cdots$}] at (0,.3) {};
		\node[draw,circle] (v3) at (1.2,.3) {$\scriptstyle{\Gamma_{x_k}}$};
		\node[ext,label=270:{\scalebox{0.5}{$1$}}] (w1) at (2.75,-1) {};
		\node[label=center:{$\scriptstyle\cdots$}] (w2) at (3.1,-1) {};
		\node[ext,label=270:{\scalebox{0.5}{$1$}}] (w3) at (3.45,-1) {};
		\node[ext] (w11) at (-2.05,-1) {};
		\node[ext] (w12) at (-1.7,-1) {};
		\node[label=center:{$\scriptstyle\cdots$}] at (-1.2,-1) {};
		\node[ext] (w13) at (-.8, -1) {};
		\node[ext] (w14) at (-.45,-1) {};
		\node[ext] (w21) at (0.35,-1) {};
		\node[ext] (w22) at (0.7,-1) {};
		\node[label=center:{$\scriptstyle\cdots$}] at (1.2,-1) {};
		\node[ext] (w23) at (1.6,-1) {};
		\node[ext] (w24) at (1.95,-1) {};
		\draw (v) edge[bend right] (v1)
		edge (v1)
		edge[bend left] (v1)
		edge[bend right] (v3)
		edge (v3)
		edge[bend left] (v3)
		edge[bend left=50] (w1)
		edge[bend left=60] (w3);
		\draw (v1)
		edge[bend right] (w11)
		edge (w12)
		edge (w13)
		edge[bend left] (w14);
		\draw (v3)
		edge[bend right] (w21)
		edge (w22)
		edge (w23)
		edge[bend left] (w24);
	\end{tikzpicture}
	\Biggr)
	\\
	&=
	\circled{\RN{1}}
	+
	\circled{\RN{2}},
\end{align*}
where $\circled{\RN{1}}$ is the sum of all terms where $d_{split}$ is applied to the internal vertices of the graph $\Gamma_v$,
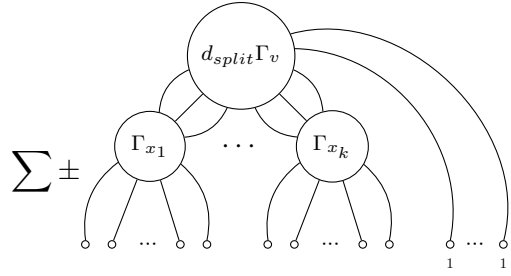
\[
\sum\pm
\begin{tikzpicture}[baseline=-.8ex]
	\node[draw,circle] (v) at (0, 1.5) {$\scriptstyle{d_{split}\Gamma_v}$};
	\node[draw,circle] (v1) at (-1.2,.3) {$\scriptstyle{\Gamma_{x_1}}$};
	\node[label=center:{$\cdots$}] at (0,.3) {};
	\node[draw,circle] (v3) at (1.2,.3) {$\scriptstyle{\Gamma_{x_k}}$};
	\node[ext,label=270:{\scalebox{0.5}{$1$}}] (w1) at (2.75,-1) {};
	\node[label=center:{$\scriptstyle\cdots$}] (w2) at (3.1,-1) {};
	\node[ext,label=270:{\scalebox{0.5}{$1$}}] (w3) at (3.45,-1) {};
	\node[ext] (w11) at (-2.05,-1) {};
	\node[ext] (w12) at (-1.7,-1) {};
	\node[label=center:{$\scriptstyle\cdots$}] at (-1.2,-1) {};
	\node[ext] (w13) at (-.8, -1) {};
	\node[ext] (w14) at (-.45,-1) {};
	\node[ext] (w21) at (0.35,-1) {};
	\node[ext] (w22) at (0.7,-1) {};
	\node[label=center:{$\scriptstyle\cdots$}] at (1.2,-1) {};
	\node[ext] (w23) at (1.6,-1) {};
	\node[ext] (w24) at (1.95,-1) {};
	\draw (v) edge[bend right] (v1)
	edge (v1)
	edge[bend left] (v1)
	edge[bend right] (v3)
	edge (v3)
	edge[bend left] (v3)
	edge[bend left=50] (w1)
	edge[bend left=60] (w3);
	\draw (v1)
	edge[bend right] (w11)
	edge (w12)
	edge (w13)
	edge[bend left] (w14);
	\draw (v3)
	edge[bend right] (w21)
	edge (w22)
	edge (w23)
	edge[bend left] (w24);
\end{tikzpicture}
\]
and $\circled{\RN{2}}$ is the remaining sum, the sum of the terms where $d_{split}$ is applied to the graphs $\Gamma_{x_1}, \ldots, \Gamma_{x_k}$.
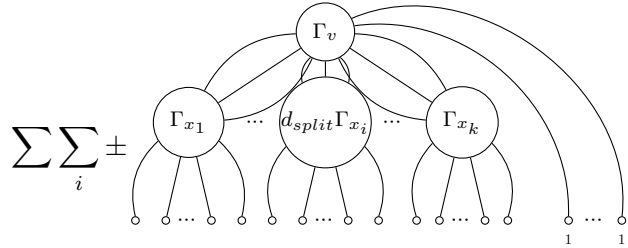
\[
\sum\sum_i\pm
\begin{tikzpicture}[baseline=-.8ex]
	\node[draw,circle] (v) at (0, 1.5) {$\scriptstyle{\Gamma_v}$};
	
	\node[draw,circle] (v1) at (-1.8,.3) {$\scriptstyle{\Gamma_{x_1}}$};
	\node[label=center:{$\scriptstyle\cdots$}] at (-0.9,.3) {}; 
	\node[draw,circle,minimum size=10mm,inner sep=0pt] (v2) 
	at (0,.3) {$\scriptstyle d_{\scriptscriptstyle split}\Gamma_{x_i}$};
	\node[label=center:{$\scriptstyle\cdots$}] at (0.9,.3) {};
	\node[draw,circle] (v3) at (1.8,.3) {$\scriptstyle{\Gamma_{x_k}}$};
	
	\node[ext,label=270:{\scalebox{0.5}{$1$}}] (w1) at (3.2,-1) {};
	\node[label=center:{$\scriptstyle\cdots$}] (w2) at (3.55,-1) {};
	\node[ext,label=270:{\scalebox{0.5}{$1$}}] (w3) at (3.9,-1) {};
	
	\node[ext] (w11) at (-2.5,-1) {};
	\node[ext] (w12) at (-2.1,-1) {};
	\node[label=center:{$\scriptstyle\cdots$}] at (-1.8,-1) {};
	\node[ext] (w13) at (-1.5, -1) {};
	\node[ext] (w14) at (-1.1,-1) {};
	
	\node[ext] (w31) at (-0.7,-1) {};
	\node[ext] (w32) at (-0.3,-1) {};
	\node[label=center:{$\scriptstyle\cdots$}] at (0.0,-1) {};
	\node[ext] (w33) at (0.3,-1) {};
	\node[ext] (w34) at (0.7,-1) {};
	
	\node[ext] (w21) at (1.2,-1) {};
	\node[ext] (w22) at (1.5,-1) {};
	\node[label=center:{$\scriptstyle\cdots$}] at (1.8,-1) {};
	\node[ext] (w23) at (2.1,-1) {};
	\node[ext] (w24) at (2.4,-1) {};
	
	\draw (v) edge[bend right] (v1)
	edge (v1)
	edge[bend left] (v1)
	edge[bend right] (v2)
	edge (v2)
	edge[bend left] (v2)
	edge[bend right] (v3)
	edge (v3)
	edge[bend left] (v3)
	edge[bend left=50] (w1)
	edge[bend left=60] (w3);
	
	\draw (v1) edge[bend right] (w11)
	edge (w12)
	edge (w13)
	edge[bend left] (w14);
	
	\draw (v2) edge[bend right] (w31)
	edge (w32)
	edge (w33)
	edge[bend left] (w34);
	
	\draw (v3) edge[bend right] (w21)
	edge (w22)
	edge (w23)
	edge[bend left] (w24);
\end{tikzpicture}
\]
Note that $\circled{\RN{1}}$ is exactly $(d_{split}\Gamma_v)\cdot (\Gamma_{x_1}; \cdots; \Gamma_{x_k})$. And with our sign conventions $\circled{\RN{2}}$ is equal to $(-1)^{|\Gamma_v|}\Gamma_v\cdot d_{split}(\Gamma_{x_1};\cdots; \Gamma_{x_k})$, which is exactly the remaining term in \eqref{eq:dc1}. This concludes the verification of \eqref{eq:dc0}.

Similarly the right action of $\hgc^\#_{U, n}$ is defined by
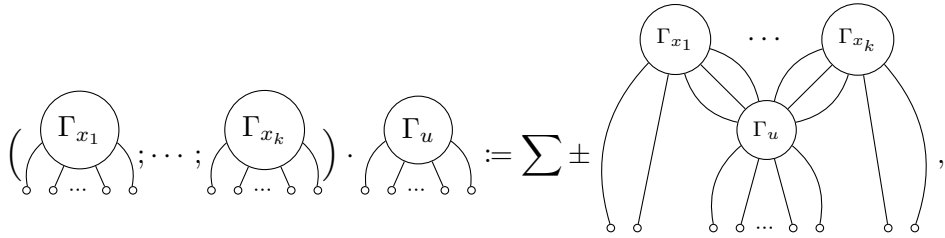
\[
\Big(
\begin{tikzpicture}[baseline=-.8ex]
\node[draw,circle] (v) at (0,.3) {$\Gamma_{x_1}$};
\node[ext] (w1) at (-.7,-.5) {};
\node[ext] (w2) at (-.35,-.5) {};
\node[label=center:{$\scriptstyle\cdots$}] (w3) at (0,-.5) {};
\node[ext] (w4) at (.35,-.5) {};
\node[ext] (w5) at (.7,-.5) {};
\draw (v) edge[bend right] (w1) edge (w2) edge (w4) edge[bend left] (w5);
\end{tikzpicture};
\cdots;
\begin{tikzpicture}[baseline=-.8ex]
\node[draw,circle] (v) at (0,.3) {$\Gamma_{x_k}$};
\node[ext] (w1) at (-.7,-.5) {};
\node[ext] (w2) at (-.35,-.5) {};
\node[label=center:{$\scriptstyle\cdots$}] (w3) at (0,-.5) {};
\node[ext] (w4) at (.35,-.5) {};
\node[ext] (w5) at (.7,-.5) {};
\draw (v) edge[bend right] (w1) edge (w2) edge (w4) edge[bend left] (w5);
\end{tikzpicture}
\Big)
\cdot
\begin{tikzpicture}[baseline=-.8ex]
\node[draw,circle] (v) at (0,.3) {$\Gamma_u$};
\node[ext] (w1) at (-.7,-.5) {};
\node[ext] (w2) at (-.35,-.5) {};
\node[label=center:{$\scriptstyle\cdots$}] (w3) at (0,-.5) {};
\node[ext] (w4) at (.35,-.5) {};
\node[ext] (w5) at (.7,-.5) {};
\draw (v) edge[bend right] (w1) edge (w2) edge (w4) edge[bend left] (w5);
\end{tikzpicture}
\coloneqq
\sum\pm
\begin{tikzpicture}[baseline=-.8ex]
\node[draw,circle] (v) at (0, 0.3) {$\scriptstyle{\Gamma_u}$};
\node[draw,circle] (v1) at (-1.2,1.5) {$\scriptstyle{\Gamma_{x_1}}$};
\node[label=center:{$\cdots$}] at (0,1.5) {};
\node[draw,circle] (v3) at (1.2,1.5) {$\scriptstyle{\Gamma_{x_k}}$};
\node[ext] (w11) at (-2.05,-1) {};
\node[ext] (w12) at (-1.7,-1) {};
\node[ext] (w23) at (1.6,-1) {};
\node[ext] (w24) at (1.95,-1) {};
\node[ext] (w01) at (-.7,-1) {};
\node[ext] (w02) at (-.35,-1) {};
\node[label=center:{$\scriptstyle\cdots$}] at (0,-1) {};
\node[ext] (w04) at (.35,-1) {};
\node[ext] (w05) at (.7,-1) {};
\draw (v1) edge (v)
edge[bend left] (v)
edge[bend right] (v)
edge[bend right] (w11)
edge (w12);
\draw (v3) edge (v)
edge[bend left] (v)
edge[bend right] (v)
edge (w23)
edge[bend left] (w24);
\draw (v) edge[bend right] (w01) edge (w02) edge (w04) edge[bend left] (w05);
\end{tikzpicture},
\]
where the sum runs over all ways to attach $\Gamma_{x_1}, \ldots, \Gamma_{x_k}$ to $\Gamma_u$ in such way that
\begin{itemize}[leftmargin=0.04\textwidth]
\item the resulting graph is internally connected;
\item all $U^*$ decorated internal vertices of $\Gamma_u$  are merged;
\item hairs of $\Gamma_{x_1}, \ldots, \Gamma_{x_i}$ can remain unattached.
\end{itemize}
A similar argument as before shows that
\[
    \hgc_{U,n}^{\#\, op} \to \Der_{L_\infty} (\hgc^\#_{U, V, n})
\]
is a Lie algebra homomorphism.

\begin{remark}
One can consider internal vertex decorations by elements of $V^*$ as "hairs sticking up", with this interpretation of internal vertex decorations the above pictures become vertically symmetric.
\end{remark}

One can easily see that both actions commute. In particular, one can use Maurer-Cartan elements $Z_u \in \MC(\hgc^\#_{U,n})$ and $Z_v \in \MC(\hgc^\#_{V,n})$ from each Lie algebra to twist the $L_\infty$-structure on $\hgc^\#_{U,V,n}$.

\subsection{Graphical composition}\label{sec:graphical composition}
Consider the three (abelian) $L_\infty$-algebras $\hgc^\#_{U,V,n}$, $\hgc^\#_{V, W,n}$ and $\hgc^\#_{U,W,n}$.

There is a morphism
\begin{equation}\label{graphical-composition-map}
    \hgc^\#_{U,V,n} \oplus \hgc^\#_{V,W,n} \to \hgc^\#_{U,W,n}
\end{equation}
of $\hgc^\#_{U,n}$-$\hgc^\#_{W,n}$-bimodules that is invariant under the diagonal action of $\hgc^\#_{V,n}$ on the source. We call it the \emph{graphical composition}.
As before, we define only the coefficients of the $L_\infty$-morphism
\[
    S^{i+j}\Big((\hgc^\#_{U,V,n} \oplus \hgc^\#_{V,W,n})[1]\Big) \to \hgc^\#_{U,W,n}[2].
\]

\[
\begin{tikzpicture}[baseline=-.8ex]
\node[draw,circle] (v) at (0,.3) {$\Gamma_{x_1}$};
\node[ext] (w1) at (-.7,-.5) {};
\node[ext] (w2) at (-.35,-.5) {};
\node[label=center:{$\scriptstyle\cdots$}] (w3) at (0,-.5) {};
\node[ext] (w4) at (.35,-.5) {};
\node[ext] (w5) at (.7,-.5) {};
\draw (v) edge[bend right] (w1) edge (w2) edge (w4) edge[bend left] (w5);
\end{tikzpicture}
\cdots
\begin{tikzpicture}[baseline=-.8ex]
\node[draw,circle] (v) at (0,.3) {$\Gamma_{x_i}$};
\node[ext] (w1) at (-.7,-.5) {};
\node[ext] (w2) at (-.35,-.5) {};
\node[label=center:{$\scriptstyle\cdots$}] (w3) at (0,-.5) {};
\node[ext] (w4) at (.35,-.5) {};
\node[ext] (w5) at (.7,-.5) {};
\draw (v) edge[bend right] (w1) edge (w2) edge (w4) edge[bend left] (w5);
\end{tikzpicture}
\cdot
\begin{tikzpicture}[baseline=-.8ex]
\node[draw,circle] (v) at (0,.3) {$\Gamma_{y_1}$};
\node[ext] (w1) at (-.7,-.5) {};
\node[ext] (w2) at (-.35,-.5) {};
\node[label=center:{$\scriptstyle\cdots$}] (w3) at (0,-.5) {};
\node[ext] (w4) at (.35,-.5) {};
\node[ext] (w5) at (.7,-.5) {};
\draw (v) edge[bend right] (w1) edge (w2) edge (w4) edge[bend left] (w5);
\end{tikzpicture}
\cdots
\begin{tikzpicture}[baseline=-.8ex]
\node[draw,circle] (v) at (0,.3) {$\Gamma_{y_j}$};
\node[ext] (w1) at (-.7,-.5) {};
\node[ext] (w2) at (-.35,-.5) {};
\node[label=center:{$\scriptstyle\cdots$}] (w3) at (0,-.5) {};
\node[ext] (w4) at (.35,-.5) {};
\node[ext] (w5) at (.7,-.5) {};
\draw (v) edge[bend right] (w1) edge (w2) edge (w4) edge[bend left] (w5);
\end{tikzpicture}
\mapsto
\sum\pm
\begin{tikzpicture}[baseline=-.8ex]
\node[draw, circle] (v1) at (-2.4,2) {$\Gamma_{y_1}$};
\node[label=center:{$\scriptstyle\cdots$}] at (-1.2,2) {};
\node[draw, circle] (v2) at (0,2) {$\Gamma_{y_{i'}}$};
\node[label=center:{$\scriptstyle\cdots$}] at ( 1.2,2) {};
\node[draw, circle] (v3) at (2.4,2) {$\Gamma_{y_i}$};
\node[draw, circle] (w1) at (-1.2,.3) {$\Gamma_{x_1}$};
\node[label=center:{$\scriptstyle\cdots$}] at (0,.3) {};
\node[draw, circle] (w2) at (1.2,0.3) {$\Gamma_{x_j}$};
\node[ext] (w11) at (-1.9,-1) {};
\node[ext] (w12) at (-1.55,-1) {};
\node[label=center:{$\scriptstyle\cdots$}] at (-1.2,-1) {};
\node[ext] (w14) at (-0.85,-1) {};
\node[ext] (w15) at (-0.5,-1) {};
\node[ext] (w21) at (0.5,-1) {};
\node[ext] (w22) at (0.85,-1) {};
\node[label=center:{$\scriptstyle\cdots$}] at (1.2,-1) {};
\node[ext] (w24) at (1.55,-1) {};
\node[ext] (w25) at (1.9,-1) {};
\node[ext,label=270:{\scalebox{0.5}{$1$}}] (v11) at (-3.1,-1) {};
\node[ext,label=270:{\scalebox{0.5}{$1$}}] (v12) at (-2.75,-1) {};
\node[ext,label=270:{\scalebox{0.5}{$1$}}] (v21) at (3.1,-1) {};
\node[ext,label=270:{\scalebox{0.5}{$1$}}] (v22) at (2.75,-1) {};
\draw (w1) edge[bend right] (w11)
           edge (w12)
           edge (w14)
           edge[bend left] (w15);
\draw (w2) edge[bend right] (w21)
           edge (w22)
           edge (w24)
           edge[bend left] (w25);
\draw (v1) edge[bend right] (v11)
           edge[bend right] (v12)
           edge (w2);
\draw (v3) edge[bend left] (v21) 
           edge[bend left] (v22);
\draw (v1) edge (w1) edge[bend right] (w1) edge[bend left] (w1);
\draw (v2) edge (w1) edge[bend right] (w1) edge[bend left] (w1)
        edge (w2) edge[bend right] (w2) edge[bend left] (w2);
\draw (v3) edge (w2) edge[bend right] (w2) edge[bend left] (w2);
\end{tikzpicture}
\]
where the sum runs over all ways to attach $\Gamma_{y_1}, \ldots, \Gamma_{y_j}$ to $\Gamma_{x_1}, \ldots, \Gamma_{x_i}$ in such a way that
\begin{itemize}[leftmargin=0.04\textwidth]
\item the resulting graph is internally connected;
\item all $V^*$ decorated internal vertices of $\Gamma_{x_1}, \ldots, \Gamma_{x_i}$ are merged;
\item all $V$ decorated hairs of $\Gamma_{y_1}, \ldots, \Gamma_{y_j}$ are attached;
\item hairs of $\Gamma_{y_1}, \ldots, \Gamma_{y_j}$ decorated by~$1$ can remain unattached.
\end{itemize}

The morphism
\[
    \hgc^\#_{U,V,n} \oplus \hgc^\#_{V,W,n} \to \hgc^\#_{U,W,n}
\]
is $\hgc^\#_{U,n}$-$\hgc^\#_{W,n}$-equivariant and invariant under the diagonal action of $\hgc^\#_{V, n}$. A similar argument to the one in~\hr[Section]{graph-action} shows that the morphism commutes with the differentials. By \hr[Corollary]{l-infty-twist-composition}, for a triple of Maurer-Cartan elements $Z_u \in \hgc^\#_{U,n}$, $Z_v \in \hgc^\#_{V,n}$ and $Z_w \in \hgc^\#_{W,n}$ there is a morphism of twisted $L_\infty$-algebras
\[
    \hgc_{U,V,n}^{\#\, Z_u, Z_v} \oplus \hgc_{V,W,n}^{\#\, Z_v, Z_w} \to \hgc_{U,W,n}^{\#\, Z_u, Z_w}.
\]

Moreover, the graphical composition is associative.

\section{Models for graphical composition}
In this section we will be freely switching between $e_n^c$-comodule $\Graphs_{U,n}^\#$ and $e_n$-module $\Graphs_{U,n}^{\#\,*}$. The reason behind is that it is much easier to give a combinatorial description of the operations for modules, but much harder to provide algebraic arguments.

We want to construct a map of simplicial sets
\[
    \alpha_{U,V}\colon \MC_{\sbullet}(\hgc^\#_{U,V,n}) \to \map_n(\Graphs^{\#\,*}_{U,n}, \Graphs^{\#\,*}_{V,n})
\]
that models the composition of mapping spaces, i.e. makes the following diagram commutative
\begin{equation}\label{composition-model}
\begin{tikzcd}
\MC_{\sbullet}(\hgc^\#_{U,V,n}) \times \MC_{\sbullet}(\hgc^\#_{V,W,n})
\arrow[r]
\arrow[d, "\alpha_{U,V} \times \alpha_{V, W}"]
&
\MC_{\sbullet}(\hgc^\#_{U,W,n})
\arrow[d, "\alpha_{U,W}"]
\\
\map_n(\Graphs^{\#\,*}_{U,n}, \Graphs^{\#\,*}_{V,n})\times\map_n(\Graphs^{\#\,*}_{V,n}, \Graphs^{\#\,*}_{W,n})
\arrow[r]
&
\map_n(\Graphs^{\#\,*}_{U,n}, \Graphs^{\#\,*}_{W,n}).
\end{tikzcd} 
\end{equation}

\begin{construction}\label{graphical-composition}
Let
$\Gamma_x =
\begin{tikzpicture}[baseline=-.8ex]
\node[draw,circle] (v) at (0,.3) {$\Gamma_x$};
\node[ext] (w1) at (-.7,-.5) {};
\node[ext] (w2) at (-.35,-.5) {};
\node[label=center:{$\scriptstyle\cdots$}] (w3) at (0,-.5) {};
\node[ext] (w4) at (.35,-.5) {};
\node[ext] (w5) at (.7,-.5) {};
\draw (v) edge[bend right] (w1) edge (w2) edge (w4) edge[bend left] (w5);
\end{tikzpicture}
\in \MC(\hgc^\#_{U, V, n})$ be a Maurer-Cartan element in $\hgc^\#_{U,V,n}$.
Define $\alpha_{U,V}(\Gamma_x) \in \mor_n(\Graphs^{\#\,*}_{U,n}, \Graphs^{\#\,*}_{V,n})$ by the formula
\[
\alpha_{U,V}(\Gamma_x)(
\begin{tikzpicture}[baseline=-.8ex]
\node[draw,circle] (v) at (0,.3) {$\gamma_u$};
\node[ext] (w1) at (-.7,-.5) {};
\node[ext] (w2) at (-.35,-.5) {};
\node[label=center:{$\scriptstyle\cdots$}] (w3) at (0,-.5) {};
\node[ext] (w4) at (.35,-.5) {};
\node[ext] (w5) at (.7,-.5) {};
\draw (v) edge[bend right] (w1) edge (w2) edge (w4) edge[bend left] (w5);
\end{tikzpicture})
\coloneqq
\sum_{i \geq 1}
\sum
\pm
\begin{tikzpicture}[baseline=-.8ex]
\node[draw,circle] (v0) at (-1.3, 1.3) {$\Gamma_x$};
\node[label=center:{$\scriptstyle\cdots$}] at (0,1.3) {};
\node[draw,circle] (v1) at ( 1.3, 1.3) {$\Gamma_x$};
\node[draw,circle] (v) at (0,0) {$\gamma_u$};
\node[ext] (w1) at (-.7,-1) {};
\node[ext] (w2) at (-.35,-1) {};
\node[label=center:{$\scriptstyle\cdots$}] at (0,-1) {};
\node[ext] (w4) at (.35,-1) {};
\node[ext] (w5) at (.7,-1) {};
\draw (v0) edge[bend right] (v) edge[bend right=60] (v) edge[bend left] (v) edge[bend left=60] (v);
\draw (v1) edge[bend right] (v) edge[bend right=60] (v) edge[bend left] (v) edge[bend left=60] (v);
\draw (v) edge[bend right] (w1) edge (w2) edge[bend left] (w2) edge (w4) edge (w5) edge[bend left] (w5);
\end{tikzpicture}
\]
where the sum runs over all ways to attach $i$ copies of $\Gamma_x$ to $\gamma_u \in \Graphs^{\#\,*}_{U,n}$ in such way that
\begin{itemize}[leftmargin=0.04\textwidth]
\item all hairs of $\Gamma_x$ are attached;
\item all $U^*$ decorated vertices of $\gamma_u$ are merged.
\end{itemize}
The resulting graph has neither decoration by $U$ nor by $U^*$, therefore $\alpha_{U,V}(\Gamma_x)(\gamma_u)$ is in $\Graphs^{\#\,*}_{V,n}$.

Note that $\alpha_{U,V} \colon \mc(\hgc^\#_{U,V,n}) \to \mor_n(\Graphs^{\#\,*}_{U,n}, \Graphs^{\#\,*}_{V,n})$ is equivariant with respect to the actions of $\hgc^\#_{U, n}$ and $\hgc^\#_{V,n}$.

Further on we will also use the dual construction with cutting graphs off (see~\cite[Definition 8.3]{will23}).
\end{construction}

Note that $\alpha_{U,V} \colon \mc(\hgc^\#_{U,V,n}) \to \mor_n(\Graphs^{\#\,*}_{U,n}, \Graphs^{\#\,*}_{V,n})$ is equivariant with respect to the actions of $\hgc^\#_{U, n}$ and $\hgc^\#_{V,n}$.

\begin{lemma}
	The diagram~\eqref{composition-model} with $\alpha_{\sbullet,\sbullet}$ defined in \textup{\hr[Construction]{graphical-composition}} commutes.
\end{lemma}

\begin{proof}
We want to show that given $\Gamma_x \in \MC(\hgc^\#_{U,V,n})$ and $\Gamma_y \in \MC(\hgc^\#_{V,W,n})$, the composition $\alpha_{V,W}(\Gamma_y)\alpha_{U,V}(\Gamma_x)$ coincides with $\alpha_{U,W}(\Gamma_y \circ \Gamma_x)$,
where $\Gamma_y \circ \Gamma_x$ is the graphical composition~\eqref{graphical-composition-map}.

Let $\gamma_u \in \Graphs^{\#\,*}_{U,n}$ as before. We split the resulting sum of $\alpha_{V,W}(\Gamma_y)\alpha_{U,V}(\Gamma_x)(\gamma_u)$ into two subsums depending on the way hairs of $\Gamma_y$ are attached to
the graphs in $\alpha_{U,V}(\Gamma_x)(\gamma_u)$.
\begin{enumerate}[leftmargin=0.04\textwidth]
\item All hairs of $\Gamma_y$ are attached to $\Gamma_x$;
\item Some hairs of $\Gamma_y$ are attached to $\gamma_u$. (Note that since $\gamma_u$ does not have any $V^*$-decorations, the only hairs of $\Gamma_y$ attached to $\gamma_u$ can be hairs decorated by $1$.)
\end{enumerate}
In the first case, the resulting graphs correspond to the graphs in $\Gamma_y \circ \Gamma_x$, where all the hairs of $\Gamma_y$ are attached to $\Gamma_x$. In the second case, the resulting graphs correspond to the graphs, where some hairs of $\Gamma_y$ decorated by~$1$ remained unattached. This concludes the proof.
\end{proof}

\subsection{Graphical category $\Gcat_n$.}\label{sec: graphical category}

We begin this section with passing to the main objects of the paper: the "unsharped" version of hairy graph complexes $\hgc_{U, n}$ and $\hgc_{U, V, n}$, and the "unsharped" version of the graph complex $\Graphs_{U,n}$.

As in \cite{will23}, we proceed in two steps. First, we twist by the Maurer-Cartan element
\[
	Z_0 =
	\begin{tikzpicture}
	\node[ext,label=270:{\scalebox{0.5}{$1$}}] (v1) at (0, -.2) {};
	\node[int] (v2) at (0,.2) {};
	
	\draw (v1) edge (v2);
	\end{tikzpicture}
	+
	\sum_{i}
	\begin{tikzpicture}
		\node[ext,label=270:{\scalebox{0.5}{$u_i$}}] (v1) at (0, -.2) {};
		\node[int,label=90:{\scalebox{0.5}{$u_i^*$}}] (v2) at (0,.2) {};
		
		\draw (v1) edge (v2);
	\end{tikzpicture}
	\in \hgc^\#_{U,n},
\]
where $u_i$'s and $u_i^*$'s are a pair of dual bases for $U$ and $U^*$, respectively. Finally, $\hgc_{U, n}$ is defined to be a dg Lie subalgebra
\[
	\hgc_{U, n} \subseteq (\hgc^\#_{U, n})^{Z_0},
\]
generated by graphs all of whose internal vertices have valence at least~$3$.

Similarly, we define $\hgc_{U,V,n}$ as an $L_\infty$ subalgebra
\[
\hgc_{U,V,n} \subseteq (\hgc^\#_{U,V,n})^{Z_0, Z_0'}
\]
generated by graphs all of whose internal vertices have valence at least~$3$. In the above, $Z_0'$ is the following Maurer-Cartan element
\[
Z_0' =
\begin{tikzpicture}
	\node[ext,label=270:{\scalebox{0.5}{$1$}}] (v1) at (0, -.2) {};
	\node[int] (v2) at (0,.2) {};
	
	\draw (v1) edge (v2);
\end{tikzpicture}
+
\sum_{i}
\begin{tikzpicture}
	\node[ext,label=270:{\scalebox{0.5}{$v_i$}}] (v1) at (0, -.2) {};
	\node[int,label=90:{\scalebox{0.5}{$v_i^*$}}] (v2) at (0,.2) {};
	
	\draw (v1) edge (v2);
\end{tikzpicture}
\in \hgc^\#_{V,n}.
\]

To define $\Graphs_{U,n}$ we need an extra piece of differential that comes from the coaction of the cooperad $\Graphs^\#_n$ on $\Graphs^\#_{U,n}$ (see~\cite[Sections~3.11, 8.5]{will23}). Namely,
\[
	Z =
	-\begin{tikzpicture}
		\node[ext] (v1) at (0, -.1) {{\scalebox{.5}{$1$}}};
		\node[int] (v2) at (0,.3) {};
		
		\draw (v1) edge (v2);
	\end{tikzpicture}
	\in \big(\Graphs_n^\#(1)\big)^*
\]
is defined to be the quotient of the twisted graph complex $(\Graphs^\#_{U,n})^{Z_0, Z}$:
\[
	\Graphs_{U,n}(r) \coloneqq (\Graphs^\#_{U,n})^{Z_0, Z}/I(r),
\]
with $I(r)$ spanned by graphs with connected components without external vertices and by graphs with vertices of valency less than 3, with a $U$-decoration contributing~$1$ to the valency.

Note that the action of $\hgc^\#_{U,n}$ on $\Graphs^\#_{U,n}$ (resp. on $\hgc_{U, V, n}$) descends to an action of $\hgc_{U,n}$ on $\Graphs_{U,n}$ (resp. on $\hgc_{U,V,n}$) (see~\cite[Section~8.3]{will23}.

\begin{remark}[see {\cite[Section 8.9]{will23}}]\label{remark:hgc for algebras}
	If $U_1 \coloneqq \Q\cdot 1 \oplus U$ has a structure of an algebra, then there is another canonical twist that is applied to the hairy graph complex. Namely,
	the element
	\[
		Z_{\mu} = \sum c_{\alpha_1, \ldots, \alpha_k, \beta}
		\begin{tikzpicture}[baseline=-.65ex]
			\node[int,label=90:{\scalebox{0.5}{$\alpha_1\cdots\alpha_k$}}] (v) at (0, .3) {};
			\node[ext,label=270:{\scalebox{0.5}{$\beta$}}] (w) at (0,-.3) {};
			\draw (v) edge (w);
		\end{tikzpicture}
		\in \hgc_{U, n}
	\]
	is a Maurer-Cartan that encodes the product in $U_1$. 
	
	Now let $A$ be an algebra, then we define
	\begin{equation}\label{eq:hgc for algebras}
		\hgc_{A, V, n} \coloneqq \hgc^{Z_{\mu}}_{\bar A, V, n}
	\end{equation}
\end{remark}

Let $\Gcat_n$ be a category enriched over simplicial sets defined by the following construction. The objects of $\Gcat$ are pairs $(U, Z_u)$ of a graded vector space $U$ and a Maurer-Cartan element $Z_u \in \hgc_{U, n}$. The mapping space is given by
\[
\map_{\sbullet}\Big((U, Z_u), (V, Z_v)\Big) \coloneqq \MC_{\sbullet}(\hgc_{U,V,n}^{Z_u, Z_v}).
\]
\hr[Corollary]{l-infty-twist-composition} equipped with the composition above implies that the composition is well-defined and associative.

\subsection{Main theorem} Finally, we are ready to state the main theorem.

\pasttheorem{main}

\begin{proof}
	Fix Maurer-Cartan elements $Z_u' = Z_u + Z_0$, $Z_v' = Z_v + Z_0'$ and $Z_w' = Z_w + Z_0''$, with 
	\[
	Z_0'' =
	\begin{tikzpicture}
		\node[ext,label=270:{\scalebox{0.5}{$1$}}] (v1) at (0, -.2) {};
		\node[int] (v2) at (0,.2) {};
		
		\draw (v1) edge (v2);
	\end{tikzpicture}
	+
	\sum_{i}
	\begin{tikzpicture}
		\node[ext,label=270:{\scalebox{0.5}{$w_i$}}] (v1) at (0, -.2) {};
		\node[int,label=90:{\scalebox{0.5}{$w_i^*$}}] (v2) at (0,.2) {};
		
		\draw (v1) edge (v2);
	\end{tikzpicture}
	\in \hgc^\#_{W,n}.
	\]
	We can apply \hr[Corollary]{l-infty-twist-composition} to the diagram~\eqref{composition-model} to get
	the following commutative diagram:
	\[
	\begin{tikzcd}
		\MC_{\sbullet}(\hgc^{Z_u,Z_v}_{U,V,n}) \times \MC_{\sbullet}(\hgc^{Z_v, Z_w}_{V,W,n})
		\arrow[r]
		\arrow[d, "\alpha_{U,V} \times \alpha_{V, W}"]
		&
		\MC_{\sbullet}(\hgc^{Z_u, Z_w}_{U,W,n})
		\arrow[d, "\alpha_{U,W}"]
		\\
		\map_n(\Graphs^{Z_u\,*}_{U,n}, \Graphs^{Z_v\,*}_{V,n})\times\map_n(\Graphs^{Z_v\,*}_{V,n}, \Graphs^{Z_w\,*}_{W,n})
		\arrow[r]
		&
		\map_n(\Graphs^{Z_u\,*}_{U,n}, \Graphs^{Z_w\,*}_{W,n}).
	\end{tikzcd}	
	\]
	Recall~\cite{will23}, that the composite
	\[
	\MC_{\sbullet}(\hgc^{Z_u,Z_v}_{U,V,n}) \xrightarrow{\alpha_{U,V}} \map_n(\Graphs^{Z_u\,*}_{U,n}, \Graphs^{Z_v\,*}_{V,n}) \to \map^h_n(\Graphs^{Z_u\,*}_{U,n}, \Graphs^{Z_v\,*}_{V,n})
	\]
	is a weak equivalence. Therefore, we can apply~\hr[Proposition]{sset-composition-model} to get the assertion.
\end{proof}

\section{Description of $\coind_n^m$}
The composition of mapping spaces of $\mathrm{dg\Lambda Hopf}$-comodules that naturally arises from the composition~\eqref{eq:emb-tilda composition} is defined in \hr[Section]{sec:general case} via $\coind$-$\cores$ adjunction. To give an explicit description of the composition we need to describe the action of the coinduction functor on the modules of configuration-space-type.

\pasttheorem{thm:coinduction graphs}

\begin{proof}
	To simplify the notation we give the proof for the untwisted case. To deal with the twisting we note that our construction below respects the action of hairy graph complexes. The transition from $Z_u$ to $Z_u^t$ happens in \eqref{eq:tree and regrading}.
		
	Since the morphism of cooperads $e_n^c \to e_m^c$ factors through the commutative cooperad $e_\infty^c \coloneqq \Com^c$ (see~\cite{fr-wi20}), one can decompose the coinduction functor as
	\[
	\coind_n^m = \coind_n^\infty\circ\coind_\infty^m.
	\]
	Therefore, we proceed in two steps:
	\begin{enumerate}[label=\textit{\roman*}), leftmargin=0.04\textwidth]
		\item \label{lcg:step one} $\coind_\infty^m(\Graphs_{U,m}) \simeq \mathbb F_{\Graphs_{U,m}^{\mathrm{tree}}(1)}$;
		\item \label{lcg:step two} $\coind_n^\infty(\mathbb F_{\Graphs_{U,m}^{\mathrm{tree}}(1)}) \simeq \Graphs_{U,n}$.
	\end{enumerate}
	To deal with~\hr[Step]{lcg:step one} we note that by~\cite{will23} $\Graphs_{U,m}$ is a $\mathrm{dg\Lambda Hopf}e_n^c$-comodule of configuration-space-type. Which means that there is a weak equivalence
	\[
	\coind_\infty^m(\Graphs_{U,m}) \simeq \mathbb F_{\Graphs_{U,m}(1)}.
	\]
	Let
	\begin{equation}\label{eq:tree inclusion}
		\Graphs_{U,m}^{\mathrm{tree}}(1) \to \Graphs_{U,m}
	\end{equation}
	a natural inclusion of the subspace spanned by trees. We claim that~\eqref{eq:tree inclusion} is a weak equivalence.
	Indeed, consider the following commutative diagram
	\[
	\begin{tikzcd}[column sep=1em]
		\Graphs_{U,m}^{\mathrm{tree}}(1)
		\arrow[rr]
		&&
		\Graphs_{U,m}(1)
		\\
		&
		U_1.
		\arrow[ur]
		\arrow[ul]
	\end{tikzcd}
	\]
	Here the diagonal arrows are given by sending $\alpha \in U_1$ to
	\begin{tikzpicture}
		\node[ext,label=-90:{$\alpha$}] (i2) at (0,.15) {};
	\end{tikzpicture}:
	a single external vertex decorated by $\alpha$. By \cite[Corollary 8.12]{will23}, both diagonal arrows are weak equivalences; hence the horizontal one is as well. In particular, the inclusion~\eqref{eq:tree inclusion} induces a weak equivalence
	\[
	\mathbb{F}_{\Graphs_{U,m}^{\mathrm{tree}}(1)}
	\to
	\mathbb{F}_{\Graphs_{U,m}(1)}
	\]
	which concludes~\hr[Step]{lcg:step one}.
	
	For \hr[Step]{lcg:step two}, we want to construct a quasi-isomorphism
	\[
	\Graphs_{U,n} \to \coind_n^\infty(\mathbb F_{\Graphs_{U,m}^{\mathrm{tree}}(1)}).
	\]
	By the $\cores$-$\coind$ adjunction, the above morphism is equivalent to a morphism of $\Com^c$-comodules
	\[
	\cores_n^\infty(\Graphs_{U,n}) \to \mathbb F_{\Graphs_{U,m}^{\mathrm{tree}}(1)}.
	\]
	We define the latter by taking projection onto the tree subspace
	\begin{equation}\label{eq:tree and regrading}
	\Graphs_{U,n}(r) \to \Graphs_{U,n}^{\mathrm{tree}}(1)^{\otimes r}
	\end{equation}
	and regrading the resulting forest. The adjoint morphism
	\[
	\Graphs_{U,n} \to \coind_n^\infty(\mathbb F_{\Graphs_{U,m}^{\mathrm{tree}}(1)})
	\]
	is a weak equivalence by the same reasoning as in~\eqref{eq:tree inclusion}. This concludes \hr[Step]{lcg:step two} and, therefore, the proof.
\end{proof}

\section{General case}\label{sec:general case}
In this section we want to modify the construction for the composition of mapping spaces to the case of hairy graph complexes corresponding to different dimensions.
Namely, we want to define a "composition"
\begin{equation}\label{dashed-composition}
	\small
	\begin{tikzcd}[column sep=0.5em]
		\map_m^h (\Graphs^{*}_{U,m}, \res^n_m\Graphs^{*}_{V,n})
		\arrow[r, phantom, "\times"]
		&
		\map_n^h (\Graphs^{*}_{V,n}, \res^k_n\Graphs^{*}_{W,k})
		\arrow[rr, dashed]
		&&
		\map_m^h (\Graphs^{*}_{U,m}, \res^k_m\Graphs^{*}_{W,k})
	\end{tikzcd}
\end{equation}
and give a graphical description in terms of the hairy graph complexes.

This is the moment when we switch to dealing comodules over cooperads rather than modules over operads as the proofs get significantly easier in this case.

Namely, in order to define the "composition"~\eqref{dashed-composition} we will define a dual map
\begin{equation}\label{diag:mapping dashed composition}
\small
\begin{tikzcd}[column sep=0.5em]
	\map_n^h(\cores_k^n(\Graphs_{W,k}), \Graphs_{V,n})
	\arrow[r, phantom, "\times"]
	&
	\map_m^h(\cores_n^m(\Graphs_{V,n}), \Graphs_{U,m})
	\arrow[rr, dashed]
	&&
	\map_m^h(\cores_k^m(\Graphs_{W,k}), \Graphs_{U,m}).
\end{tikzcd}
\end{equation}

We define the dashed arrow above via the following diagram
\[
\footnotesize
\begin{tikzcd}[column sep=0.5em]
	\map_n^h(\cores_k^n(\Graphs_{W,k}), \Graphs_{V,n})
	\arrow[r, phantom, "\times"]
	\arrow[d, equal]
	&
	\map_m^h(\cores_n^m(\Graphs_{V,n}), \Graphs_{U,m})
	\arrow[rr, dashed]
	\arrow[d, "\cong"]
	&&
	\map_m^h(\cores_k^m(\Graphs_{W,k}), \Graphs_{U,m})
	\arrow[d, "\cong"]
	\\
	\map_n^h(\cores_k^n(\Graphs_{W,k}), \Graphs_{V,n})
	\arrow[r, phantom, "\times"]
	&
	\map_n^h\big(\Graphs_{V,n}, \coind_n^m(\Graphs_{U,m})\big)
	\arrow[rr]
	&&
	\map_n^h\big(\cores_k^n(\Graphs_{W,k}), \coind_n^m(\Graphs_{U,m})\big),
\end{tikzcd}
\]
where the vertical isomorphisms are given by $\cores$-$\coind$ adjunction, and the bottom horizontal arrow is a natural composition in the category of $\mathrm{dgHopf}e_n^c$-modules-comodules.

In order to pass to hairy graph complexes in the bottom row, we want to ensure that the target comodules are of configuration-space type. Consider the following diagram:
\begin{equation}\label{diag:adjunction composition}
\scriptsize
\begin{tikzcd}[column sep=0.5em]
	\map_n^h(\cores_k^n(\Graphs_{W,k}), \Graphs_{V,n})
	\arrow[r, phantom, "\times"]
	\arrow[d, "\cong"]
	&
	\map_n^h\big(\Graphs_{V,n}, \coind_n^m(\Graphs_{U,m})\big)
	\arrow[rr]
	\arrow[d, equal]
	&&
	\map_n^h\big(\cores_k^n(\Graphs_{W,k}), \coind_n^m(\Graphs_{U,m})\big)
	\arrow[d, "\cong"]
	\\
	\map_k^h\big(\Graphs_{W,k}, \coind_k^n(\Graphs_{V,n})\big)
	\arrow[r, phantom, "\times"]
	\arrow[d, equal]
	&
	\map_n^h\big(\Graphs_{V,n}, \coind_n^m(\Graphs_{U,m})\big)
	\arrow[rr, dashed]
	\arrow[d, "\coind_k^n"]
	&&
	\map_k^h\big(\Graphs_{W,k}, \coind_k^m(\Graphs_{U,m})\big)
	\arrow[d, equal]
	\\
	\map_k^h\big(\Graphs_{W,k}, \coind_k^n(\Graphs_{V,n})\big)
	\arrow[r, phantom, "\times"]
	&
	\map_k^h\big(\coind_k^n(\Graphs_{V,n}), \coind_k^m(\Graphs_{U,m})\big)
	\arrow[rr]
	&&
	\map_k^h\big(\Graphs_{W,k}, \coind_k^m(\Graphs_{U,m})\big),
\end{tikzcd}
\end{equation}
where the top vertical isomorphisms come from the $\cores$–$\coind$ adjunction, while the bottom arrow is given by $\coind_k^n$.
The dashed arrow is defined via surrounding isomorphisms to make both squares commute.

\begin{lemma}
	The diagram~\eqref{diag:adjunction composition} commutes.
\end{lemma}
\begin{proof}
	Follows directly from \hr[Proposition]{adjunction-composition-compatibility}.
\end{proof}

\vspace{.5em}
In the light of~\hr[Proposition]{thm:coinduction graphs} the two bottom rows of the diagram~\eqref{diag:adjunction composition} can be rewritten as
\[
\begin{tikzcd}[column sep=0.5em]
	\map_k^h(\Graphs_{W,k}, \Graphs_{V,k})
	\arrow[r, phantom, "\times"]
	\arrow[d, equal]
	&
	\map_n^h(\Graphs_{V,n}, \Graphs_{U,n})
	\arrow[rr, dashed]
	\arrow[d]
	&&
	\map_k^h(\Graphs_{W,k}, \Graphs_{U,k})
	\arrow[d, equal]
	\\
	\map_k^h(\Graphs_{W,k}, \Graphs_{V,k})
	\arrow[r, phantom, "\times"]
	&
	\map_k^h(\Graphs_{V,k}, \Graphs_{U,k})
	\arrow[rr]
	&&
	\map_k^h(\Graphs_{W,k}, \Graphs_{U,k}).
\end{tikzcd}
\]

Now we pass to the hairy graph complexes. The map:
\[
	\alpha_{U,V}\colon \MC_{\sbullet}(\hgc_{U, V, n}) \to \map_n(\Graphs_{V,n}, \Graphs_{U,n})
\]
is dual to the one described in \hr[Construction]{graphical-composition}.
Our goal is to describe the map
\[
	\MC_{\sbullet}(\hgc_{U,V,n}) \to \MC_{\sbullet}(\hgc_{U,V,k})
\]
that makes the diagram
\begin{equation}\label{diag:hairy graphs to mapping space} 
\begin{tikzcd}
	\MC_{\sbullet}(\hgc_{U,V,n})
	\arrow[r, "\alpha^h_{U,V}"]
	\arrow[d]
	&
	\map^h_n(\Graphs_{V,n}, \Graphs_{U,n})
	\arrow[d]
	\\
	\MC_{\sbullet}(\hgc_{U,V,k})
	\arrow[r, "\alpha^h_{U,V}"]
	&
	\map^h_k(\Graphs_{V,k}, \Graphs_{U,k})
\end{tikzcd}
\end{equation}
commute.

We begin with a sequence of maps between (non-derived) mapping spaces.

\begin{equation}\label{diag:hairy graphs to mapping space-expanded}
\begin{tikzcd}
	\MC_{\sbullet}(\hgc_{U,V,n})
	\arrow[r, "\alpha_{U,V}"]
	&
	\map_n(\Graphs_{V,n}, \Graphs_{U,n})
	\arrow[d, "\coind_k^n"]
	\\
	&
	\map_k(\coind_k^n \Graphs_{V,n}, \coind_k^n \Graphs_{U,n})
	\arrow[d]
	\\
	&
	\map_k(\coind_k^\infty \coind_\infty^n \Graphs_{V,n}, \coind_k^\infty \coind_\infty^n \Graphs_{U,n})
	\arrow[d]
	\\
	&
	\map_k(\coind_k^\infty \mathbb F_{\Graphs_{V,n}^{\mathrm{tree}}(1)}, \coind_k^\infty \mathbb F_{\Graphs_{U,n}^{\mathrm{tree}}(1)})
\end{tikzcd}
\end{equation}

The image of $\alpha_{U,V}(\Gamma_x)\in \map_n(\Graphs_{V,n}, \Graphs_{U,n})$ in
$\map_k(\coind_k^\infty \mathbb F_{\Graphs_{V,n}^{\mathrm{tree}}(1)}, \coind_k^\infty \mathbb F_{\Graphs_{U,n}^{\mathrm{tree}}(1)})$ is given on the set of cogenerators by cutting the graph and evaluating multiples of copies of $\Gamma_x$ on the top piece.
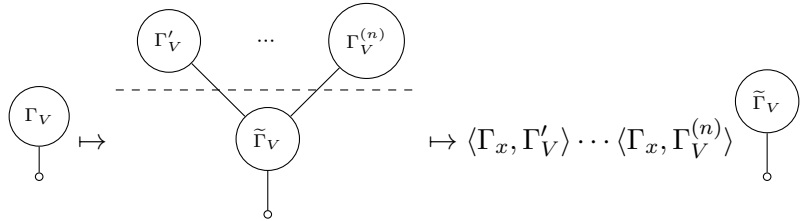
\[
\begin{tikzpicture}[baseline=-.8ex]
	\node[draw,circle] (v) at (0,.3) {\tiny$\Gamma_V$};
	\node[ext] (w1) at (0,-.5) {};
	\draw (v) edge(w1);
\end{tikzpicture}
\mapsto
\begin{tikzpicture}[baseline=-.8ex]
	\node[draw,circle] (v0) at (-1.3, 1.3) {\tiny$\Gamma_V'$};
	\node[label=center:{$\scriptstyle\cdots$}] at (0,1.3) {};
	\node[draw,circle] (v1) at ( 1.3, 1.3) {\tiny$\Gamma_V^{(n)}$};
	\node[draw,circle] (v) at (0,0) {\tiny$\widetilde \Gamma_V$};
	\node[ext] (w1) at (0,-1) {};
	\draw (v0) edge (v);
	\draw (v1) edge (v);
	\draw (v) edge (w1);
	\draw[dashed] (-2, {0.5*1.3}) -- (2, {0.5*1.3});
\end{tikzpicture}
\mapsto
\langle \Gamma_x, \Gamma_V'\rangle \cdots \langle \Gamma_x, \Gamma_V^{(n)}\rangle
\begin{tikzpicture}[baseline=-.8ex]
	\node[draw,circle] (v) at (0,0.5) {\tiny$\widetilde\Gamma_V$};
	\node[ext] (w1) at (0,-0.5) {};
	\draw (v) edge (w1);
\end{tikzpicture}
\]
In particular, the result is not zero only on tree part of $\Gamma_x$.
Therefore, the composite
$\MC_{\sbullet}(\hgc_{U, V, n}) \to 
\map_k(\coind_k^\infty \mathbb F_{\Graphs_{V,n}^{\mathrm{tree}}(1)}, \coind_k^\infty \mathbb F_{\Graphs_{U,n}^{\mathrm{tree}}(1)})$
factors through $\MC_{\sbullet}(\hgc_{U,V,n}^{\mathrm{tree}})$.

\begin{lemma}
	The following diagram is commutative
	\[
	\begin{tikzcd}
		\MC_{\sbullet}(\hgc_{U,V,n})
		\arrow[r, "\alpha_{U,V}"]
		\arrow[ddd]
		&
		\map_n(\Graphs_{V,n}, \Graphs_{U,n})
		\arrow[d, "\coind_k^n"]
		\\
		&
		\map_k(\coind_k^n \Graphs_{V,n}, \coind_k^n \Graphs_{U,n})
		\arrow[d]
		\\
		&
		\map_k(\coind_k^\infty \coind_\infty^n \Graphs_{V,n}, \coind_k^\infty \coind_\infty^n \Graphs_{U,n})
		\arrow[d]
		\\
		\MC_{\sbullet}(\hgc^{\mathrm{tree}}_{U,V,n})
		\arrow[r, "\alpha_{U,V}"]
		&
		\map_k(\coind_k^\infty \mathbb F_{\Graphs_{V,n}^{\mathrm{tree}}(1)}, \coind_k^\infty \mathbb F_{\Graphs_{U,n}^{\mathrm{tree}}(1)}),
	\end{tikzcd}
	\]
	where the left vertical arrow is given by projection onto the tree part.\qed
\end{lemma}

The final piece of the description is in the following lemma with a straightforward proof.

\begin{lemma}
	Let $\MC_{\sbullet}(\hgc_{U, V, n}^{\mathrm{tree}}) \to \MC_{\sbullet}(\hgc_{U, V, k})$ be the regrading map. Then there is a commutative diagram:
	\[
	\begin{tikzcd}
	\MC_{\sbullet}(\hgc^{\mathrm{tree}}_{U,V,n})
	\arrow[r, "\alpha_{U,V}"]
	\arrow[dd]
	&
	\map_k(\coind_k^\infty \mathbb F_{\Graphs_{V,n}^{\mathrm{tree}}(1)}, \coind_k^\infty \mathbb F_{\Graphs_{U,n}^{\mathrm{tree}}(1)}),
	\arrow[d]
	\\
	&
	\map_k(\Graphs_{V,k}, \coind_k^\infty \mathbb F_{\Graphs_{U,n}^{\mathrm{tree}}(1)})
	\\
	\MC_{\sbullet}(\hgc_{U,V,k})
	\arrow[r, "\alpha_{U,V}"]
	&
	\map_k\big(\Graphs_{V,k}, \Graphs_{U,k}\big)
	\arrow[u]
	\end{tikzcd}
	\]	
\end{lemma}

To summarise, the map
\[
\MC_{\sbullet}(\hgc_{U,V,n}) \to \MC_{\sbullet}(\hgc_{U,V,k})
\]
that models $\coind_k^n$ is given by composite of projection onto the tree part and regrading.

Now we can define a model for the middle composition in~\eqref{diag:adjunction composition} via the following diagram
\[
\begin{tikzcd}
	\MC_{\sbullet}(\hgc_{V,W,k})
	\arrow[d, equal]
	\arrow[r, phantom, "\times"]
	&
	\MC_{\sbullet}(\hgc_{U,V,n})
	\arrow[d]
	\arrow[r]
	&
	\MC_{\sbullet}(\hgc_{U,W,k})
	\arrow[d, equal]
	\\
	\MC_{\sbullet}(\hgc_{V,W,k})
	\arrow[r, phantom, "\times"]
	&
	\MC_{\sbullet}(\hgc_{U,V,k})
	\arrow[r]
	&
	\MC_{\sbullet}(\hgc_{U,W,k}).
\end{tikzcd}
\]
Namely, the map
\[
	\MC_{\sbullet}(\hgc_{V,W,k}) \times \MC_{\sbullet}(\hgc_{U,V,n})
	\to
	\MC_{\sbullet}(\hgc_{U,W,k})
\]
is given by the following procedure. For $\Gamma_{x_1}, \ldots, \Gamma_{x_i} \in \hgc_{U,V,n}$ we take the projection onto the tree part, then regrade the result, and, finally, attach $\Gamma_{y_1}, \ldots, \Gamma_{y_j} \in \hgc_{V,W,k}$ to it in all possible ways:
\[
\begin{tikzpicture}[baseline=-.8ex]
	\node[draw,circle] (v) at (0,.3) {$\Gamma_{x_1}$};
	\node[ext] (w1) at (-.7,-.5) {};
	\node[ext] (w2) at (-.35,-.5) {};
	\node[label=center:{$\scriptstyle\cdots$}] (w3) at (0,-.5) {};
	\node[ext] (w4) at (.35,-.5) {};
	\node[ext] (w5) at (.7,-.5) {};
	\draw (v) edge[bend right] (w1) edge (w2) edge (w4) edge[bend left] (w5);
\end{tikzpicture}
\cdots
\begin{tikzpicture}[baseline=-.8ex]
	\node[draw,circle] (v) at (0,.3) {$\Gamma_{x_i}$};
	\node[ext] (w1) at (-.7,-.5) {};
	\node[ext] (w2) at (-.35,-.5) {};
	\node[label=center:{$\scriptstyle\cdots$}] (w3) at (0,-.5) {};
	\node[ext] (w4) at (.35,-.5) {};
	\node[ext] (w5) at (.7,-.5) {};
	\draw (v) edge[bend right] (w1) edge (w2) edge (w4) edge[bend left] (w5);
\end{tikzpicture}
\cdot
\begin{tikzpicture}[baseline=-.8ex]
	\node[draw,circle] (v) at (0,.3) {$\Gamma_{y_1}$};
	\node[ext] (w1) at (-.7,-.5) {};
	\node[ext] (w2) at (-.35,-.5) {};
	\node[label=center:{$\scriptstyle\cdots$}] (w3) at (0,-.5) {};
	\node[ext] (w4) at (.35,-.5) {};
	\node[ext] (w5) at (.7,-.5) {};
	\draw (v) edge[bend right] (w1) edge (w2) edge (w4) edge[bend left] (w5);
\end{tikzpicture}
\cdots
\begin{tikzpicture}[baseline=-.8ex]
	\node[draw,circle] (v) at (0,.3) {$\Gamma_{y_j}$};
	\node[ext] (w1) at (-.7,-.5) {};
	\node[ext] (w2) at (-.35,-.5) {};
	\node[label=center:{$\scriptstyle\cdots$}] (w3) at (0,-.5) {};
	\node[ext] (w4) at (.35,-.5) {};
	\node[ext] (w5) at (.7,-.5) {};
	\draw (v) edge[bend right] (w1) edge (w2) edge (w4) edge[bend left] (w5);
\end{tikzpicture}
\mapsto
\sum\pm
\begin{tikzpicture}[baseline=-.8ex]
	\node[draw, circle] (v1) at (-2.4,2) {$\Gamma_{y_1}$};
	\node[label=center:{$\scriptstyle\cdots$}] at (-1.2,2) {};
	\node[draw, circle] (v2) at (0,2) {$\Gamma_{y_{i'}}$};
	\node[label=center:{$\scriptstyle\cdots$}] at ( 1.2,2) {};
	\node[draw, circle] (v3) at (2.4,2) {$\Gamma_{y_i}$};
	\node[draw, circle] (w1) at (-1.2,.3) {\tiny$\mathrm{tree}(\Gamma_{x_1})$};
	\node[label=center:{$\scriptstyle\cdots$}] at (0,.3) {};
	\node[draw, circle] (w2) at (1.2,0.3) {\tiny$\mathrm{tree}(\Gamma_{x_i})$};
	\node[ext] (w11) at (-1.2,-1) {};
	\node[ext] (w21) at (1.2,-1) {};
	\node[ext,label=270:{\scalebox{0.5}{$1$}}] (v11) at (-3.1,-1) {};
	\node[ext,label=270:{\scalebox{0.5}{$1$}}] (v12) at (-2.75,-1) {};
	\node[ext,label=270:{\scalebox{0.5}{$1$}}] (v21) at (3.1,-1) {};
	\node[ext,label=270:{\scalebox{0.5}{$1$}}] (v22) at (2.75,-1) {};
	\draw (w1) edge (w11);
	\draw (w2) edge (w21);
	\draw (v1) edge[bend right] (v11)
	edge[bend right] (v12)
	edge (w2);
	\draw (v3) edge[bend left] (v21) 
	edge[bend left] (v22);
	\draw (v1) edge (w1) edge[bend right] (w1) edge[bend left] (w1);
	\draw (v2) edge (w1) edge[bend right] (w1) edge[bend left] (w1)
	edge (w2) edge[bend right] (w2) edge[bend left] (w2);
	\draw (v3) edge (w2) edge[bend right] (w2) edge[bend left] (w2);
\end{tikzpicture}.
\]

\pasttheorem{main-general}

\begin{proof}
Since the horizontal maps in the diagram
\[
\begin{tikzcd}
\MC_{\sbullet}(\hgc^{Z_v, Z_w}_{V,W,k})
\arrow[r, phantom, "\times"]
\arrow[d, "\alpha_{U,V}"]
&
\MC_{\sbullet}(\hgc^{Z_u, Z_v}_{U,V,n})
\arrow[r]
\arrow[d, "\alpha_{V,W}"]
&
\MC_{\sbullet}(\hgc^{Z_u, Z_w}_{U,W,k})
\arrow[d, "\alpha_{V,W}"]
\\
\map_k(\Graphs^{Z_w}_{W,k}, \Graphs^{Z_v}_{V,k})
\arrow[r, phantom, "\times"]
&
\map_n(\Graphs^{Z_v}_{V,n}, \Graphs^{Z_u}_{U,n})
\arrow[r]
&
\map_k(\Graphs^{Z_w}_{W,k}, \Graphs^{Z_u}_{U,k})
\end{tikzcd}
\]
were built to make the diagram commute, the assertion follows from \hr[Proposition]{sset-composition-model}.
\end{proof}

\appendix
\section{Compatibility of Adjunctions with Composition}

\begin{proposition}\label{adjunction-composition-compatibility}
Let $L\colon \mathcal C \leftrightarrows \mathcal D\rcolon R$ be a pair of adjoint functors. Then the diagram
\[
\begin{tikzcd}
	\mor_{\mathcal C}(X, Y)
	\arrow[r, phantom, "\times"]
	\arrow[d, "L"]
	&
	\mor_{\mathcal C}(Y, RZ)
	\arrow[r]
	\arrow[d, "\cong"]
	&
	\mor_{\mathcal C}(X, RZ)
	\arrow[d, "\cong"]
	\\
	\mor_{\mathcal D} (LX, LY)
	\arrow[r, phantom, "\times"]
	&
	\mor_{\mathcal D}(LY, Z)
	\arrow[r]
	&
	\mor_{\mathcal D}(LX, Z)
\end{tikzcd}
\]
commutes. Here, the first vertical arrow is given by application of $L$, the remaining vertical arrows are the isomorphisms of the adjunction, and the horizontal arrows are natural composition of morphisms.
\end{proposition}

\begin{proof}
The desired diagram is a "concatenation" of two commutative squares
\[
\begin{tikzcd}
	\mor_{\mathcal C}(X, Y)
	\arrow[r, phantom, "\times"]
	\arrow[d, "L"]
	&
	\mor_{\mathcal C}(Y, RZ)
	\arrow[r]
	\arrow[d, "L"]
	&
	\mor_{\mathcal C}(X, RZ)
	\arrow[d, "L"]
	\\
	\mor_{\mathcal D} (LX, LY)
	\arrow[r, phantom, "\times"]
	\arrow[d, equal]
	&
	\mor_{\mathcal D}(LY, LRZ)
	\arrow[r]
	\arrow[d, "(\varepsilon_Z)_*"]
	&
	\mor_{\mathcal D}(LX, LRZ)
	\arrow[d, "(\varepsilon_Z)_*"]
	\\
	\mor_{\mathcal D}(LX, LY)
	\arrow[r, phantom, "\times"]
	&
	\mor_{\mathcal D}(LY,Z)
	\arrow[r]
	&
	\mor_{\mathcal D}(LX, Z),
\end{tikzcd}
\]
where $(\varepsilon_Z)_*\colon LRZ \to Z$ is the adjunction counit. Thus, the outer square is commutative.
\end{proof}

\end{document}